\documentclass{amsart}
\usepackage[intlimits,sumlimits,noDcommand,narrowiints]{kpfonts}
\usepackage[english]{babel}
\usepackage{mathtools}
\usepackage{amsthm}
\usepackage{slashed}
\usepackage{bm}
\usepackage{comment}
\usepackage{tikz-cd}
\usepackage[pagewise,mathlines]{lineno}
\usepackage{fullpage}
\theoremstyle{plain}
\newtheorem{thm}{Theorem}[section]

\newtheorem{prop}[thm]{Proposition}
\newtheorem{lem}[thm]{Lemma}
\newtheorem{cor}[thm]{Corollary}
\theoremstyle{definition}

\theoremstyle{remark}
\newtheorem{rmk}[thm]{Remark}
\newcommand\beq{\begin{equation}}
\newcommand\eeq{\end{equation}}
\newcommand\beqs{\begin{equation*}}
\newcommand\eeqs{\end{equation*}}
\newcommand\St{{\Sigma_\tau}}
\newcommand\h{\hspace}
\newcommand\mcC{\mathcal C}

\newcommand*\R{\mathbb{R}}
\newcommand*\D{\mathrm{d}}
\newcommand*\M{\mathrm{m}}
\newcommand*\p{\partial}

\DeclarePairedDelimiterX\set[2]\lbrace\rbrace{#1\;\delimsize\vert\;#2}
\newcommand*\eqdef{\overset{\mbox{\tiny{def}}}{=}}
\DeclarePairedDelimiter{\lrangle}{\langle}{\rangle}

\newcommand*\mr{\mathring}
\newcommand{\fk}{\mathfrak}

\DeclareMathOperator{\dvol}{dvol}

\DeclareMathOperator{\foc}{foc}
\DeclareMathOperator{\diag}{diag}

\declareslashed{}{/}{.1}{0}{\xi}
\newcommand*\Lebw{\mathcal{L}}
\newcommand*\Sobw{\mathcal{W}}

\numberwithin{equation}{section}
 
\begin{document}
\title{Global stability of some totally geodesic wave maps}
\author{Leonardo Enrique Abbrescia}
\address{Vanderbilt University}
\email{leonardo.abbrescia@vanderbilt.edu}
\author{Yuan Chen}
\address{Michigan State University}
\email{chenyu60@msu.edu}


\begin{abstract}
We prove that wave maps that factor as $\mathbb{R}^{1+d} \overset{\varphi_{\text{S}}}{\to} \mathbb{R} \overset{\varphi_{\text{I}}}{\to} M$, subject to a sign condition, are globally nonlinear stable under small compactly supported perturbations when $M$ is a spaceform. The main innovation is our assumption on $\varphi_{\text{S}}$, namely that it be a semi-Riemannian submersion. This implies that the background solution has infinite total energy, making this, to the best of our knowledge, the first stability result for factored wave maps with infinite energy backgrounds. We prove that the equations of motion for the perturbation decouple into a nonlinear wave--Klein-Gordon system. We prove global existence for this system and improve on the known regularity assumptions for equations of this type.
\end{abstract}

\maketitle

\section{Introduction} \label{intro}
This paper is concerned with the global stability of certain infinite energy totally geodesic wave maps from Minkowski space $\mathbb{R}^{1+d}$ with $d \geq 3$ into a spaceform $(M^n,g)$. We consider as our background solutions totally geodesic maps which factor as 
\begin{equation} \label{eq:factor}
\begin{tikzcd}
\R^{1+d} \arrow{r}{\varphi_{\text{S}}} & \R \arrow{r}{\varphi_{\text{I}}} & M;
\end{tikzcd}
\end{equation}
where, denoting by $e$ the standard Euclidean metric on $\mathbb{R}$, the mapping $\varphi_\text{S}$ is a semi-Riemannian submersion\footnote{A semi-Riemannian submersion $\varphi : N \to M$ is necessarily an isometry on the horizontal space normal to fibres. See \cite[P. 212]{Oneill} for a precise definition.} to either $(\mathbb{R},e)$ or $(\mathbb{R},-e)$, and $\varphi_{\text{I}}$ is a Riemannian immersion from $(\mathbb{R},e)$ to $(M,g)$. In particular, this factorization implies the background solution is a totally geodesic wave map  that has infinite total energy. The semi-Riemannian submersion $\varphi_{\text{S}}$ can be classified as \emph{spacelike} or \emph{timelike}\footnote{Note that by definition, a semi-Riemannian submersion cannot be null. We always equip the real line $\R$, as the domain of $\varphi_{\text{I}}$, with $+e$.} depending on whether its codomain $\mathbb{R}$ is considered as being equipped with $e$ or $-e$. Our main theorem states:

\begin{thm}[Rough version] \label{roughthm}
    Fix $d \geq 3$. A totally geodesic map satisfying the factorization \eqref{eq:factor} is globally nonlinearly stable as a solution to the initial value problem for the wave maps equation under compactly supported smooth perturbations, provided that either
\begin{description}
    \item[TL] $\varphi_{\text{S}}$ is timelike and $(M,g)$ is a negatively-curved spaceform;
    \item[SL] $\varphi_{\text{S}}$ is spacelike and $(M,g)$ is a positively-curved spaceform.     \end{description}
\end{thm}

This paper is organized as follows: we first give a brief discussion of harmonic and wave maps and describe our motivation to study the problem at hand by drawing comparisons to some known results. The rest of the introduction is dedicated to providing an explanation of Theorem \ref{roughthm}, including  our approach for setting up the problem and highlighting potential difficulties and how we overcome them. Section \ref{geometry} is dedicated to geometric preliminaries, which we use in Section \ref{equations} to derive the equations of motion for the perturbations. Next we set up the analytical framework for our proof in Section \ref{section:tools}. Finally, Sections \ref{section:negcurv} and \ref{section:negpos} are dedicated to the proof of Theorem \ref{roughthm} in the settings of {\textbf{TL}} and {\textbf{SL}}, respectively. 

\subsection{Some background and motivation for our problem} A map $\phi : N \to M$ between two Riemannian manifolds $(N,h)$ and $(M,g)$ is said to be \emph{harmonic} if it is a critical point of the action\footnote{In the forthcoming discussion, the word ``energy'' does not represent the action $\mathcal{S}[\phi]$. Rather, it will always refer to the $L^2$ energy on spacelike hypersurfaces of $\R^{1+d}$.}
\begin{equation}
\mathcal{S}[\phi] \eqdef \frac{1}{2} \int_N \langle \D \phi, \D \phi\rangle_{T^*N\otimes \phi^{-1}TM} \ \dvol_h.
\end{equation}
In local coordinates on $(M,g)$, the Euler-Lagrange equations (ELE) take the form
\begin{equation}\label{EL-Harmonic}
\Delta_h \phi^i + \Gamma_{jk}^i(\phi) \langle \D \phi^j, \D \phi^k\rangle_h = 0.
\end{equation}
Here $\Delta_h$ is the Laplace-Beltrami operator on $N$ and $\Gamma_{jk}^i(\phi)$ are the Christoffel symbols of $M$ evaluated along the image of $\phi$.

In their famous paper \cite{eells1964harmonic}, Eells and Sampson showed that if $\phi$ is a harmonic map from a compact manifold $N$ with nonnegative Ricci curvature to a target $M$ with nonpositive sectional curvature, then it is also \emph{totally geodesic}\footnote{A map $\phi: N \to M$ is totally geodesic if it maps every geodesic of $N$ onto a geodesic of $M$.}. On the other hand, it follows from a straightforward computation that every {totally geodesic} map is harmonic. Moreover, a well known result of Vilms shows that ``if $N$ is complete, then every totally geodesic map $\phi : N \to M$ {factors} as 
\begin{equation}\label{eq:factorTG}
\begin{tikzcd}
N \arrow{r}{\Phi_{\text{S}}} & B \arrow{r}{\Phi_{\text{I}}} & M;
\end{tikzcd}
\end{equation}
with $\Phi_{\text{S}}$ a Riemannian submersion and $\Phi_{\text{I}}$ a Riemannian immersion, both being totally geodesic" \cite{vilms1970totally}. This motivates our setting where our background solutions are of the form \eqref{eq:factor}, as factored totally geodesic maps of the form \eqref{eq:factorTG} are a well-studied class of harmonic maps. 

Of course, our source is the \emph{Lorentzian} manifold $\R^{1+d}$ equipped with the standard Mikowski metric $\M = \diag(-1,1,\dots,1)$. In the case that the source metric $h$ is Lorentzian, the elliptic equation \eqref{EL-Harmonic} becomes a hyperbolic equation, and analyzing the solution $\phi$ amounts to an initial value problem. In this setting we say a solution to the ELE is a \emph{wave map} and, in the case where $(N,h) = (\R^{1+d},\M)$, we look for a solution to the equation\footnote{In this paper greek indices run $\mu = (0,1,\dots,d)$ and index coordinates on $\R^{1+d}$. Latin indices run $i = (1,\dots,n)$ and index local coordinates on the target $M$. Repeated indices will always be summed using Einstein summation notation. We also use the short hand $\partial_\mu = \tfrac{\p}{\p x^\mu}$ and sometimes identify $t = x^0$.}
\begin{equation}\label{wavemaps}
\Box_\M \phi^i + \Gamma_{jk}^i (\phi) \M^{\mu\nu} \partial_\mu \phi^j \partial_\nu\phi^k = 0.
\end{equation}
The theory of wave maps has a rich history, for a general review see \cite{WMBook, krieger2007global}. For now we remark that our results apply to the physically relevant cases where the target $M^n = \mathbb S^n$ or $\mathbb H^n$. The former case models the nonlinear sigma model in plasma physics \cite{gell1960axial}, while the latter has applications in general relativity \cite{Bruhat}. 

Since our composed totally geodesic map $\varphi_{\text{I}}\circ \varphi_{\text{S}}$ is also a wave map, we study its stability under small initial perturbations in the Cauchy problem for \eqref{wavemaps}. We emphasize that, as $\varphi_{\text{S}} : \R^{1+d} \to \R$ is an orthogonal projection onto a 1-dimensional subspace, it automatically satisfies the linear wave equation. Furthermore, the total geodesy of the composed map trivially implies that the image of $\varphi_{\text{I}}$ is a geodesic in $M$. 

Sideris and Grigoryan have previously studied stability of factored (non-totally geodesic) wave maps
\[\begin{tikzcd}
\R^{1+d} \arrow{r}{\varphi_{\text{W}}} & \R \arrow{r}{\varphi_{\text{G}}} & M
\end{tikzcd}\]
\cite{Sideris, Grigoryan}. In his paper, Sideris was motivated to study the stability of wave maps \emph{localized} to a geodesic to overcome singularity issues discovered in \cite{shatah}, where singular solutions for the nonlinear $\sigma$-model $\R^{1+3} \to \mathbb S^3$ were constructed whose range contained a hemisphere. Our problem is related to \cite{Sideris, Grigoryan} in that their background is also the composition of a geodesic $\varphi_{\text{G}}$ and a solution to the linear wave equation $\varphi_{\text{W}}$. Contrastingly, their $\varphi_{\text{W}}$ is an arbitrary \emph{finite energy} solution to the linear wave equation and hence $\varphi_{\text{G}}\circ\varphi_{\text{W}}$ is {not} totally geodesic. This provides yet another motivation for our problem where we  assume that $\varphi_{\text{S}}$ is assumed to be a semi-Riemannian submersion and hence has \emph{infinite} total energy. This introduces considerable difficulties as the finite energy backgrounds of \cite{Sideris, Grigoryan} \emph{decay} at the expected rate of finite-energy waves, whereas ours are non-decaying. 

\subsection{Explanation of results} In this subsection we clarify the geometric set-up for Theorem \ref{roughthm} and expand on the precise analytical difficulties and conclusions of the result.   

In this paper we adapt the geometric framework of \cite{Sideris, Grigoryan}, where we write the equations of motion for the perturbation in a tubular neighborhood $\R \times \mathcal{N}$ of the geodesic $\varphi_{\text{I}}(\R)\subset M$ (here $\R$ parametrizes the geodesic and $\mathcal{N}$ the normal $(n-1)$-directions). The main geometric contribution of this paper is Proposition \ref{decoupledeqs}, which shows that the equations for the perturbation $\bm u = (u^1,\vec u \hspace{.05cm}) \in \R \times \mathcal{N}$ decouple into a system of wave and Klein-Gordon equations: 
\begin{equation} \label{W-KGsystem}
\begin{cases}
\Box u^1 =  F^1\bm u \cdot  \M(\D  \bm u,\D \varphi_{\text{S}}) + O(|\bm u|^3 + |\partial\bm u|^3),  \\
\Box \vec{u} - \vec{M} \vec{u}  = \vec F\bm u \cdot \M(\D  \bm u,\D \varphi_{\text{S}}) + O(|\bm u|^3 + |\partial \bm u|^3).
\end{cases}
\end{equation}
Here $F^1, \ \vec F$ are functions of the curvature of $(M,g)$ restricted to the geodesic $\varphi_{\text{I}}$. The $\vec M$ are the masses of $\vec u$, and as a consequence of the spaceform assumption on $M$, Proposition \ref{decoupledeqs} implies $\vec M = \kappa \M(\D \varphi_{\text{S}},\D \varphi_{\text{S}})$ where $\kappa$ is the sectional curvature of $M$. Hence, the assumptions on $\varphi_{\text{S}}$ in Theorem \ref{roughthm} are there to at minimum guarantee linear stability, i.e. make the Klein-Gordon terms $\vec u$ have positive masses. 

The computations leading to Proposition \ref{decoupledeqs} and \eqref{W-KGsystem} hinge on a careful Taylor expansion of the Christoffel symbols $\Gamma$ about the geodesic $\varphi_{\text{I}}(\R)$. This is where our geometric approach differs from that of \cite{Sideris, Grigoryan}. In their works, the authors need only perform a rough quadratic Taylor expansion because they are able to utilize the decay properties of their background. In the current paper, we perform a precise \emph{cubic} expansion to capture the lowest order nonlinear structures. As we will see, our precise control on these Taylor coefficients reveal weak null-structures that prevent resonant interactions that could lead to finite-time blow up, see \eqref{weaknull} and Lemma \ref{christoffder2}. Finally, we remark that the geometry of the target manifold $M$ in \cite{Sideris, Grigoryan} is arbitrary. Morally, the premise for their stability result is that their background solution converges to the same point in $M$ (as a consequence of finite energy!) as one moves in any direction on $\R^{1+d}$ to infinity. As our background is not decaying, moving along generic directions on $\R^{1+d}$ does not imply that the image $\varphi_{\text{I}}\circ \varphi_{\text{S}}$ in $M$ converges to a single point. Our spaceform assumption is then a natural way to ensure some sort of homogeneity of the geometry of $M$ as one moves towards infinity on $\R^{1+d}$ along the mapping $\varphi_{\text{I}}\circ \varphi_{\text{S}}$. 

\begin{rmk}\label{generalization}
As we will see, for the energy estimates of higher derivatives of $\bm u$ we need first and second order commutations of the equations \eqref{W-KGsystem} with the Lorentz boosts $L^i = t \p_{x^i} + x^i\p_t$. Under the spaceform assumption, the functions $F$ and $\vec F$ are \emph{constant} and hence vanish when differentiated. In the case that the curvature is \emph{not} constant, these coefficients can grow: $L^i F \approx t (F')$. Using the weak null structures revealed in \eqref{weaknull} and \eqref{weaknull2}, each order of Taylor expansions introduce an additional Klein-Gordon factor (which has a linear decay rate of $|\vec u \hspace{.05cm}| \lesssim t^{-d/2}$):
\[ F = \sum_{|\alpha| \le N} \p^\alpha F(0) (\vec u \hspace{.05cm})^{N} + O(|\vec u\hspace{.05cm}|^{N+1}).\]
As $d \ge 3$, this decay can \emph{overcome} the aforementioned growth. And, consequently, we can easily relax the spaceform assumption to targets $(M,g)$ with the following property: \emph{along $\varphi_{\text{I}}(\R)$, the metric $g$ agrees with a spaceform up to fourth order}. See also Remark \ref{generalization2}
\end{rmk}

The main analytic contributions of the present paper are Theorems \ref{maintheorem1} and \ref{maintheorem2}, which provide an open set (in a suitable Sobolev topology) of initial data such that the Cauchy problem for \eqref{W-KGsystem} has a global solution in spatial dimension $ d= 3$. For our analysis of the equations of motion we use the physical space vectorfield method and its related energy estimates. 

\begin{rmk}[Dimensionality and strongly coupled nonlinearities]
We restrict the proof of the main theorem and the discussions below to $d = 3$ because stability of quadratic wave--Klein-Gordon systems is a known standard result in dimensions $d \ge 4$, see \cite{sogge}. This leaves the stability problem for the class of totally geodesic wave maps studied here open when $d = 2$.

Nonlinear wave--Klein-Gordon systems in two spatial dimensions have seen dramatic progress in recent years \cite{ifrim2019global, WKG2D, ma2020global, duan2020global, dong2021stability}. The nonlinear interactions present in \eqref{W-KGsystem} fall into a broader class called \emph{strongly coupled nonlinearities}, initially proposed (but not proven to be stable) by Ma in \cite{WKG2D}. The difficulties associated to strong coupling are intimately related to a breaking of conformal invariance of the wave equation; when this conformal invariance is available, it leads to stronger dispersive decay \cite{Wong2017}. Nevertheless, in the recent celebrated works of \cite{ma2020global, duan2020global}, Ma and Duan were able to prove a preliminary stability result to the wave--Klein-Gordon system \eqref{W-KGsystem} in two space dimensions. Finally, the incredibly interesting recent work \cite{dong2021stability} of Dong and Wyatt gives a full resolution of the global stability of totally geodesic maps satisfying \eqref{eq:factor} in two space dimensions. Their new ideas include an $L^1$ version of a conformal energy estimate and a key divergence structure in the equations \eqref{W-KGsystem} to apply a normal form transformation. This transformation turns the critical quadratic nonlinearities into cubic ones which are no longer critical. Their purely physical space approach produces uniform control for the low derivative energy levels.
\end{rmk}

We note that, for $d = 3$, global existence of coupled wave and Klein-Gordon equations is known, see the monograph by LeFloch and Ma \cite{LeFloch}. In the present manuscript we give a short proof of their result using the hyperboloidal method developed in \cite{Wong2017}, which is a geometric refinement of \cite{Klaine1985a, LeFloch}. We remark that in that same article, Wong proved global existence for \eqref{wavemaps} for all $d \ge 2$ where $\phi$ is a small perturbation of a constant.  Our main analytical tools are Wong's geometric formulation of the weighted $L^2$---$L^\infty$ Sobolev inequalities adapted to the Lorentz boosts and the interpolated versions of the Gagliardo-Nirenberg-Sobolev inequality adapted to the hyperboloids developed by Wong and the first author \cite{GNAWS}. This allows us to avoid using the purely spatial rotations, and as our most important analytic contribution, to prove stability of \eqref{W-KGsystem} assuming that the initial data is in $H^3$, see Remark \ref{regularity}. To the best of the author's knowledge, the best prior results in $d = 3$ using purely physical space techniques for wave--Klein-Gordon systems was stability with initial data at the level of $H^6$ \cite{LeFloch}.

\begin{rmk}
We note that there are technical differences between the systems studied in \cite{LeFloch} and \eqref{W-KGsystem}. LeFloch and Ma considered a \emph{quasilinear} system of wave--Klein-Gordon equations, which introduces additional difficulties. On the other hand, their nonlinearities satisfy the \emph{classical} null condition of Klainerman \cite{Klaine1984a} which allows them to extract \emph{improved} decay from all quadratic nonlinearities. We emphasize that our nonlinearities do not satisfy the classical null condition, and hence we are not able to extract the improved decay present in \cite{LeFloch}.
\end{rmk}

\begin{rmk}[Regularity] \label{regularity}
To guarantee global existence it suffices that the initial perturbation is sufficiently small in $H^3\times H^2$; this level of smallness is enough to guarantee $C^1$ convergence. Note that a standard persistence of regularity argument implies that if initial data is in $H^4\times H^3$ with \emph{smallness} in $H^3\times H^2$, this guarantees that the solutions remain small in $H^3\times H^2$, converges to 0 in $C^1$, and has bounded $C^2$ norm globally. As we will see, pushing the regularity down to $H^3\times H^2$ requires our bootstrap mechanism to allow for \emph{growth} in the top order energies, see Proposition \ref{prop:bootstrap}. Roughly speaking, this is because the improved linear decay for Klein-Gordon derivatives $|\p \vec u\hspace{.05cm} | \lesssim t^{-3/2}$ is available using only the \emph{third} order energies (at the level of $H^4$). Instead, by sacrificing a decay factor of $t^{-1/2}$, we can rely on the interpolated Sobolev embeddings of \cite{GNAWS} to close the argument at the level of $H^3$. Had we assumed \emph{smallness} in $H^4$, our arguments could easily be adapted to prevent this growth, guaranteeing $C^2$ convergence. 
\end{rmk}

What makes our argument run through is that the spaceform curvature restrictions expose hidden weak null structures that make harmful wave--wave resonant terms from the quadratic and cubic nonlinearities vanish. More precisely, we will show that the \emph{undifferentiated} factor $u^1$ is missing in $\bm u\cdot \M(\D  \bm u, \D \varphi_{\text{S}})$ and $|\bm u|^3$ in equations \eqref{W-KGsystem}. We also show that the quadratic nonlinearity for the wave solution $u^1$ is of the form $\vec{u}\cdot \p \vec{u}$. 

There are numerous ramifications of these exposed null conditions. Firstly, as we are unable to use the Morawetz vectorfield as a multiplier, the available decay rate for $u^1$ in dimension 3 is $t^{-1/2}$. This means that terms of the form $(u^1)^2$ or $(u^1)^3$ (which are excluded by our exposed null structures) could lead to finite-time blow up.  Secondly, it is crucial that \emph{only} $\vec{u}$ appear in the quadratic nonlinearity for $u^1$ because its expected decay rate is $|\vec{u}\hspace{.05cm} | + |\p\vec{u}\hspace{.05cm}| \lesssim t^{-3/2}$, compared to the derivative wave decay $|\p u^1| \lesssim t^{-1}$. This improved decay for the nonlinearity of $u^1$ will feed-back into the Klein-Gordon equations when we try to estimate $|\vec{u}\cdot \M(\D u^1,\D \varphi_{\text{S}})|$, allowing us to close our estimates.

We conclude by stating that this geometric hyperboloidal method has previously been used to prove global nonlinear stability for certain infinite energy solutions to \emph{quasilinear} wave equations. The first author and Wong \cite{AbbWong2019} used these techniques to prove that the membrane equation 
\begin{equation}\label{membrane}
0 = \p_\mu\left( \frac{\M^{\mu\nu} \p_\nu\phi}{\sqrt{1 + \M(\D \phi,\D\phi)}}\right)
\end{equation}
admits an open set of initial data such that $\phi$ is a small perturbation of an infinite energy simple-plane-wave.

\section{Geodesic normal coordinates} \label{geometry}

In this section we set up the geometric tools and notations needed for the rest of the manuscript. We first consider the case where $M$ is an arbitrary complete Riemannian manifold and later specialize to the spaceform setting.

We will construct a system of coordinates for a tubular neighborhood of an arbitrary geodesic, in which the restriction of the Christoffel symbols to the geodesic vanish. For a comprehensive treatment of the geometry of geodesic normal coordinates, see the book by Alfred Gray \cite{Gray}. He analyzes a generalization of geodesic normal coordinates called \emph{Fermi coordinates}. They give a local description of a tubular neighborhood about an embedded submanifold $P \subset M$ of arbitrary codimension.

Consider our complete Riemannian manifold $ (M^n,g)$ and let $\gamma : \R \to M$ be a fixed geodesic parametrized by arc-length. Let $\mathcal{V} = \{ (\gamma(t), v) \ | \ t \in \R, \ v \in T_{\gamma(t)}M^\perp\}$ denote the normal bundle along the geodesic $\gamma$. We write $\mathcal{V}_\gamma$ for the fibres above $\gamma$ and $\mathcal{V}_{\gamma(t_0)}$ when we wish to specify the fibre above a specific point $\gamma(t_0)$. We now construct an explicit local orthonormal frame of a subbundle of $\mathcal{V}$ and use it to define the so called geodesic normal coordinates by the exponential map.

We will parametrize the tubular neighborhood of $\gamma$ by $\R \times \mathcal{N}$, where 
\[\mathcal{N} \eqdef \{ \vec{x} = (x^2,\dots,x^n) \in \R^{n-1} \ | \ |\vec{x} \hspace{.05cm}| <  r_{\foc}(\gamma)\}.\]
Here $r_{\foc}(\gamma)$ is the \emph{focal radius of $\gamma$}, which is defined to be the maximal radius such that the normal exponential map, see \eqref{eq:def:normalexp}, is non-critical on the normal disc bundle of $\gamma$ of radius $r_{\foc}(\gamma)$. In the subsequent analysis of the wave map problem, we can guarantee that $r_{\foc}(\gamma) > 0$ because of the spaceform assumption.
\begin{rmk}\label{generalization2}
The generalization of Remark \ref{generalization} to targets $(M,g)$ such that the metric agrees with a spaceform up to fourth order along the geodesic $\varphi_{\text{I}}(\R)$ should also be accompanied with the following assumptions:
\begin{itemize}
\item the focal radius $r_{\foc}(\gamma)$ is bounded away from zero;
\item higher derivatives of the metric \emph{in geodesic normal coordinates} are bounded away from infinity.
\end{itemize}
\end{rmk}

Denote $e_1 \eqdef \dot \gamma(0)$ and use it to define an orthonormal basis 
\[ e^\perp \eqdef ({e_2,\dots,e_n})\] 
of $\mathcal{V}_{\gamma(0)}$. For arbitrary $x^1 \in \R$ let $(e_1(x^1), e^\perp(x^1))$ be defined by parallel transporting $(e_1,e^\perp)$ along $\gamma$ and note that $e^\perp(x^1)$ is an orthonormal frame for $\mathcal{V}_{\gamma(x^1)}$. Also note that $e_1(x^1) = \dot \gamma(x^1)$ by definition.

For $(x^1, \vec x \hspace{.05 cm}) \in \R \times \mathcal N$, define $\gamma^\perp(x^1; \vec{x},s)$ as the unique geodesic (with path parameter $s$) defined by 
\[ \gamma^\perp(x^1;\vec{x},0) = \gamma(x^1),\qquad \dot\gamma^\perp(x^1; \vec{x},0) = \sum_{k=2}^n x^k e_k(x^1).\]
\begin{rmk}\label{rmk:originalgeo}
We identify the original geodesic $\gamma$ with $\{\vec{x} \equiv 0\}$ because we have $\gamma^\perp(x^1;0,s) = \gamma^\perp(x^1;0,0) = \gamma(x^1)$ for any $s \in [0,1]$ by uniqueness of ODEs. 
\end{rmk}
% \begin{rmk}
% Note that $\Span\{e^\perp(x^1)\} = \mathcal{V}_{\gamma(x^1)}$ and $\Span\{e_1(x^1), e^\perp(x^1)\} = T_{\gamma(x^1)}M$.
% \end{rmk}
 We can now define the normal exponential map
\begin{equation} \label{eq:def:normalexp}
 \exp^\perp_{\gamma(x^1)}(\vec{x}\hspace{.05cm}) \eqdef \gamma^\perp(x^1,\vec{x},1),
\end{equation}
 which is a map
 \[ \exp^\perp_{\gamma(x^1)} : \mathcal{V}_{\gamma(x^1)} \to M.\] 
This normal exponential map shares many features with the usual one from Riemannian geometry. For example, the inverse function theorem and the following computation show that $\mathcal{N}$ is non-trivial and that $\exp^\perp_{\gamma(\cdot)}(\cdot)$ is indeed a smooth immersion from $\R \times \mathcal{N}$ to a a tubular neighborhood, which we denote as $\mathcal{T}$, around $\gamma \subset M$:
\begin{equation} \label{expder}
\begin{aligned}
\D \left(\exp^\perp_{\gamma(x^1)}\right)_{\{ \vec{x} = 0\}} (\vec{y}) & = \frac{d}{ds} \bigg|_{s = 0} \exp^\perp_{\gamma(x^1)}( s \vec{y})   = \frac{d}{ds} \bigg|_{s = 0} \gamma(x^1; s\vec{y},1) \\
& =  \frac{d}{ds} \bigg|_{s = 0} \gamma(x^1; \vec{y},s)\\
& = \vec{y}.
\end{aligned}
\end{equation}
\begin{rmk} \label{expder1}
Of course, the ``$\D$" in \eqref{expder} denoted the differential of the map in the $\vec{x}$ variables. Since $\exp^\perp_{\gamma(\cdot)}(\cdot)$ is technically a map on domain $\R \times \mathcal N$, in the future we write $\D_{x^1}$ as the differential in the $x^1$ variable. It is easy to see that, restricted to $\{\vec{x} = 0\}$, 
\[ \D_{x^1}\left(\exp^\perp_{\gamma(y^1)}\right)_{0} (e_1) = e_1(y^1).\]
\end{rmk}

 This allows us to define the geodesic normal coordinates by the preimage of the exponential map 
\begin{center}
\begin{tikzcd}
\mathcal{T} \arrow{r}{\exp^{-1}_\gamma} & \R \times \mathcal{N}.
\end{tikzcd}
\end{center}
More explicitly, if \[ q = \exp_{\gamma(x^1)}(\vec{x}) \in \mathcal{T},\] then $q$ can be written in geodesic normal coordinates by  $(x^1,\vec{x}\hspace{.05cm}) = (\exp^\perp)^{-1}_\gamma(q)$.
\begin{rmk}
Here $(\exp^\perp_\gamma)^{-1}$ is the pre-image of the exponential map. Technically, $\exp^\perp_\gamma$ it is not a bona fide diffeomorphism near self intersections of $\gamma$. In the case that $\gamma$ is an embedded geodesic, then $\exp^\perp_\gamma$ is a true diffeomorphism. As we are in the perturbative regime for the ensuing analysis of the wave maps equation, our considerations are local so we will ignore any self intersections.

\end{rmk}
%\begin{rmk}
%We abuse notation and identify the metric $g$ on $\mathcal{T}\subset M$ with the pull-back metric given by $\exp_\gamma^*g$ on some open subset of $\R\times\mathcal{N}$. We also identify $\exp_{\gamma(x^1)}(\vec{x})$ with $(x^1,\vec{x})$ and 
%\end{rmk}
The following lemma sets up the key geometric tools that we need for our analysis. Even though the proof is standard, we include it for the sake of completion:
 
\begin{lem} \label{geotools} Let $\tfrac{\p}{\p x^i}, \ i = 1,\dots,n$ be the coordinate vectorfields defined by $(x^1,\vec{x})$. Let $y^1 \in \R$ be arbitrary. Then the following identities hold
\begin{align}
\frac{\partial}{\partial x^i}\bigg|_{\gamma(y^1)} & = e_i(y^1)  & i \in\{1,\dots,n\}  \label{coordvect} \\
g_{ij}\bigg|_{\gamma(y^1)} & = \delta_{ij}  & i,j \in\{1,\dots,n\}\label{deltametric} \\
\p_k g_{ij} \bigg|_{\gamma(y^1)} & = 0   & i,j,k \in\{1,\dots,n\} \label{metricder} \\
\nabla_{\p_{x^i}} \p_{x^j}\bigg|_{\gamma(y^1)} & = 0  & i,j \in\{1,\dots,n\}  \label{covder} \\
\Gamma_{ij}^k\bigg|_{\gamma(y^1)} & = 0  & i,j,k \in\{1,\dots,n\} \label{christoff} \\
 \label{christoffder} \p_m \Gamma_{ij}^k \bigg|_{\gamma(y^1)} &=  \p_m \lrangle{\nabla_{\p_{x^i}}\p_{x^j},\p_{x^k}}_g \bigg|_{\gamma(y^1)}& i,j,k,m \in \{1,\dots,n\} \\
\label{christoff2der} \p^2_{mp} \Gamma_{ij}^k \bigg|_{\gamma(y^1)} &=  \p^2_{mp} \lrangle{\nabla_{\p_{x^i}}\p_{x^j},\p_{x^k}}_g \bigg|_{\gamma(y^1)}& i,j,k,m,p \in \{1,\dots,n\}
\end{align}
\end{lem}
\begin{proof}
The proof of \eqref{coordvect} for $i \in \{2,\dots,n\}$ follows by definition and \eqref{expder}
\[\frac{\p}{\p x^i} \bigg|_{\gamma(y^1)} = \D \left(\exp_{\gamma(y^1)}\right)_{\vec{x} = 0}\left(e_i(y^1)\right) = e_i(y^1).\]
The case of $i = 1$ follows from the discussion in Remark \ref{expder1}: 
\[ \frac{\p}{\p x^1}\bigg|_{\gamma(y^1)} = e_1(y^1).\]
Equation \eqref{deltametric} follows immediately because $\{e_i(y^1) \ | \ i = 1,\dots, n\}$ were defined by parallel transporting an orthonormal set and because parallel transport is an isometry. 

The definition of parallel transport implies  
\[ \nabla_{\p_{x^1}} \p_{x^1} \bigg|_{\gamma(y^1)} = \nabla_{\p_{x^1}} \p_{x^j} \bigg|_{\gamma(y^1)} = 0 \]
for $j \in \{2,\dots,d\}$. Because we are using coordinate vectorfields and $\nabla$ is torsion free,
\[ \nabla_{\p_{x^m}} \p_{x^k} = \nabla_{\p_{x^k}} \p_{x^m}\]
even away from $\gamma$. This implies 
\[ \nabla_{\p_{x^j}} \p_{x^1} \bigg|_{\gamma(y^1)} = 0.\] 
Next we note that any $X \in \mathcal{V}_{\gamma(y^1)}$ is tangent to the curve $\exp_{\gamma(y^1)}(sX)$. Since this curve is a geodesic by definition, we have that $\nabla_X X = 0$. In particular we have that 
\[ 0 = \nabla_{\p_{x^i} + \p_{x^j}}(\p_{x^i} + \p_{x^j}) = \nabla_{\p_{x^i}}\p_{x^j} + \nabla_{\p_{x^j}}\p_{x^i}\]
whenever $i,j \in \{ 2,\dots,n\}$. The torsion condition and setting $s = 0$ proves
\[ \nabla_{\p_{x^i}} \p_{x^j} \bigg|_{\gamma(y^1)} = 0.\] 
This concludes the proof of \eqref{covder}.

Using that the Levi-Civita connection is metric, we see
\[ \p_{x^k} g_{ij} = \p_{x^k} \lrangle{\p_{x^i},\p_{x^j}}_g = \lrangle{\nabla_{x^k}\p_{x^i},\p_{x^j}}_g + \lrangle{\p_{x^i},\nabla_{x^k}\p_{x^j}}_g.\] 
Restricting this computation to $\gamma$ concludes the proof of \eqref{metricder} using \eqref{covder}.

For the Christoffel symbols, recall that they are defined by 
\[ \nabla_{\p_{x^i}}\p_{x^j} = \Gamma_{ij}^m \p_{x^m}.\]
Taking the inner product with $\p_{x^k}$, using the metric structure \eqref{deltametric}, and \eqref{covder} proves
\[ 0 = \lrangle{\nabla_{\p_{x^i}}\p_{x^j},\p_{x^k}}_g \bigg|_{\gamma(y^1)} = \Gamma_{ij}^k\bigg|_{\gamma(y^1)}.\]

Finally, we compute 
\[ \p_m  \lrangle{\nabla_{\p_{x^i}}\p_{x^j},\p_{x^k}}_g = \p_m ( \Gamma_{ij}^l g_{lk}) = \p_m \Gamma_{ij}^l g_{lk} + \Gamma_{ij}^l \p_m g_{lk}.\]
This and \eqref{deltametric}, \eqref{metricder} show \eqref{christoffder}. Similarly, \eqref{christoff2der} follows.

\end{proof}

\section{Perturbed system and Reduction to Wave--Klein-Gordon system} \label{equations}
We now return to the wave map equation and describe the precise construction for the perturbation to our totally geodesic background 
\[\begin{tikzcd}
\R^{1+d} \arrow{r}{\varphi_{\text{S}}} & \R \arrow{r}{\varphi_{\text{I}}} & M.
\end{tikzcd}\]
Recall that $\varphi_{\text{S}}$ is a semi-Riemannian submersion, so in accordance with the discussion in Section \ref{intro} regarding \cite{vilms1970totally}, we prescribe $\varphi_{\text{S}}$ to be a \emph{linear} function $\ell: \R^{1+d} \to \R$ satisfying $\M(\D \ell,\D \ell) = \pm 1$. As $\varphi_{\text{I}}$ is an immersed geodesic in $M$, we identify it with the zero cross-section about the normal bundle of $\varphi_{\text{I}}(\R) \subset M$ (see Remark \ref{rmk:originalgeo}):
\[\begin{tikzcd}
 \R \arrow[hook]{r}{\iota} & \R \times \mathcal{N} \arrow{r}{\exp^\perp_{\varphi_{\text{I}}}} & M;
\end{tikzcd}\]
where the first map is the inclusion and the second map is the restriction $\exp^\perp_{\varphi_{\text{I}}} \big|_{\vec x = 0}$.

Equipping $\R \times \mathcal{N}$ with the pull-back metric, we then look for maps of the form $\phi \eqdef \iota \circ \ell + \bm u$ which are solutions to the wave maps equation \eqref{wavemaps} on\footnote{As we will see, initial data for $\bm u$ can be chosen small enough so that $\iota \circ \ell + \bm u = (\ell + u^1, \vec u \hspace{.05cm}) \in \R \times \mathcal{N}$.}
\[ \begin{tikzcd}
\R^{1+d} \arrow{r}{\iota \circ \ell + \bm u\ } & \R \times \mathcal{N}.
\end{tikzcd}\]
Here addition is taken coordinate wise on $\R \times \mathcal{N}$. Consequently, the perturbation of our totally geodesic background takes the form
\[\begin{tikzcd}
\exp^\perp_{\varphi_{\text{I}}}(\iota \circ \ell + \bm u) : \R^{1+d} \arrow{r} & M
\end{tikzcd}\]
which is also a solution to the wave maps equation. Of course, $\bm u \equiv 0$ corresponds to the background $\varphi_{\text{I}} \circ \varphi_{\text{S}}$. As we consider $\ell$ fixed, the equations of motion \eqref{wavemaps} for $\phi$ reduce to a Cauchy problem for the perturbation $\bm u$.

Let $(x^1,\dots,x^n)$ be the geodesic normal coordinates about $\varphi_{\text{I}}$ constructed in section \ref{geometry}. In these coordinates $\bm u = (u^1,u^2,\dots,u^n) = (u^1,\vec{u}\hspace{.05cm})$ and hence $\phi$ takes the form 
\[ \phi = (\ell + u^1,u^2,\dots,u^n) = (\ell + u^1,\vec{u}\hspace{.05cm}).\] 
The equations of motion \eqref{wavemaps} take the form
\begin{equation}\label{perturbation}
\begin{aligned}
\Box u^1 + \Gamma_{jk}^1\left(\ell + u^1,\vec{u}\right) \cdot \M(\D \phi^j,\D \phi^k) & = 0, & \\
\Box u^i + \Gamma_{jk}^i\left(\ell + u^1,\vec{u}\right) \cdot \M(\D \phi^j,\D \phi^k) & = 0 ,& i \in \{2,\dots,n\}.
\end{aligned}
\end{equation}
We compute
\[\M(\D \phi^j,\D \phi^k) = \M(\D \ell,\D \ell) \delta_1^j \delta_1^k + \M(\D \ell,\D u^j)\delta_1^k + \M(\D \ell,\D u^k)\delta_1^j + \M(\D u^j,\D u^k).\]
Taylor expanding $\Gamma$ about the geodesic $\varphi_{\text{I}}\circ \ell$ we see
\[\Gamma_{jk}^i(\ell + u^1,\vec{u}) = \Gamma_{jk}^i(\ell,\vec{0}) + \sum_{m=1}^n\p_m \Gamma_{jk}^i(\ell,\vec{0})u^m + O(|\bm u|^2).\]
We pause at this juncture to make some reductions. From \eqref{christoff} we see that the first term on the right hand side vanishes. Moreover, since $\Gamma_{jk}^i(\ell, \vec{0}) = 0$ for arbitrary $\ell$, we see that 
\begin{equation}
\underbrace{\p_1\cdots\p_1}_{\text{$q$ times}} \Gamma_{jk}^i(\ell, \vec{0}) = 0  \label{weaknull}
\end{equation}
for all $i,j,k \in \{1,\dots,n\}$ and any positive integer $q$.

Before we expand the Christoffel symbols up to third order, we introduce the following notation: if $A$ is an $m$-tuple with elements drawn from $\{1,\dots,n\}$ (namely that $A = (A_1,\dots,A_m)$ with $A_i \in \{1,\dots,n\}$), for a scalar function $f$ we denote
\[ \p_A f \eqdef \p_{x^{A_1}}\cdots\p_{x^{A_n}} f.\]
By $|A|$ we refer to its length, namely $m$. Given a vector $\bm x = (x_1,\dots,x_n) \in \R^n$, we denote
\[ \bm x^A \eqdef x^{A_1}\cdots x^{A_m}.\]
We also introduce the shorthand 
\begin{equation}\label{sizeL}
\sigma \eqdef \M(\D\ell,\D\ell)
\end{equation}
to denote the size of $\D\ell$ as measured by the Minkowski metric. Expanding out the Christoffel symbols up to third order in $\bm u$, we see that the equations can be expressed as 

\begin{multline} \label{TaylorW}
\Box u^1 + \sum_{m=2}^n \p_m\Gamma_{11}^1(\ell,\vec{0}) u^m \cdot\sigma  =  \\
- 2 \sum_{m=2}^d  \p_m\Gamma_{j1}^1(\ell,\vec{0}) u^m \cdot \M(\D u^j,\D \ell) - \sum_{\substack{|A| = 2 \\ A \neq (1,1)}} \p_A \Gamma_{11}^1(\ell,\vec{0}) \bm u^A \cdot \sigma \\
- \sum_{m=2}^n\p_m\Gamma_{jk}^1(\ell,\vec{0}) u^m \cdot \M(\D u^j,\D u^k) - 2 \sum_{\substack{|A| = 2  \\ A \neq (1,1)}} \p_A \Gamma_{j1}^1(\ell,\vec 0)\bm u^A  \cdot \M(\D u^j, \D \ell) \\
-  \sum_{\substack{|A| = 3 \\ A \neq (1,1,1)}}^n \p_A \Gamma_{11}^1(\ell,\vec{0}) \bm u^A \cdot\sigma + \text{h.o.t.},
\end{multline}
\begin{multline}\label{TaylorKG}
\Box u^i + \sum_{m=2}^n \p_m\Gamma_{11}^i(\ell,\vec{0}) u^m \cdot\sigma  = \\
- 2 \sum_{m=2}^d  \p_m\Gamma_{j1}^i(\ell,\vec{0}) u^m \cdot \M(\D u^j,\D \ell) - \sum_{\substack{|A| = 2 \\ A \neq (1,1)}} \p_A \Gamma_{11}^i(\ell,\vec{0}) \bm u ^A \cdot \sigma \\
- \sum_{m=2}^n\p_m\Gamma_{jk}^i(\ell,\vec{0}) u^m \cdot \M(\D u^j,\D u^k) - 2 \sum_{\substack{|A| = 2 \\ A \neq (1,1)}} \p_A \Gamma_{j1}^i(\ell,\vec 0) \bm u^A \cdot \M(\D u^j, \D \ell) \\
-  \sum_{\substack{|A| = 3 \\ A \neq (1,1,1)}} \p_A \Gamma_{11}^i(\ell,\vec{0}) \bm u^A \cdot\sigma + \text{h.o.t.}
\end{multline} 
\begin{rmk}
To clarify, the sums involving $A$ on the right hand side of \eqref{TaylorW} and \eqref{TaylorKG} are summing over m-tuples $A = (A_1,\dots,A_m)$ \emph{excluding} the vertex $A_i = 1$ for all $i$. That is, for example,
\[ \sum_{\substack{|A| = 2  \\ A \neq (1,1)}} \p_A \Gamma_{11}^1 \bm u^\alpha \eqdef \sum_{\substack{\alpha_1, \alpha_2 = 1 \\ (\alpha_1,\alpha_2) \neq (1,1)}}^n \p_{x^{\alpha_1}}\p_{x^{\alpha_2}}\Gamma_{11}^1 u^{\alpha_1} \cdot u^\alpha_2.\]
That we are able to do this is of course a consequence of \eqref{weaknull}.
\end{rmk}
\begin{rmk}
As stated previously, repeated latin indices are implicitly summed over $\{1,\dots,n\}$ unless otherwise stated. For example, 
\[ \sum_{m= 2}^n \p_m \Gamma_{j1}^i u^m \cdot \M(\D u^j,\D \ell) \eqdef \sum_{\substack{m = 2 \\ j = 1}}^n \p_m \Gamma_{j1}^i u^m \cdot \M(\D u^j,\D \ell).\]
\end{rmk}

\begin{rmk}
In the equations ``h.o.t.'' represents higher order terms of the form
\[ \text{h.o.t.} \lesssim C_M (|\bm u|^4 + |\p \bm u|^4)\cdot f(\bm u,\p \bm u),\]
where $C_M$ denotes some constant depending on the derivatives of the Christoffel symbols of the target manifold restricted to the geodesic. Here $f: \R^{n(d+2)} \to \R$ is a smooth function depending on $M$.
\end{rmk}

We are able to find explicit formulas for the coefficients of the linear terms:
\begin{lem} \label{christoffder2}
Let $i,k,m \in \{1,\dots,n\}$. Then, restricted to the geodesic ${\varphi_{\text{I}}}$, we have
\[ \p_m \Gamma_{k1}^i \bigg|_{\varphi_{\text{I}}} = R_{m1ki}\bigg|_{\varphi_{\text{I}}}.\]
\end{lem}
\begin{proof}
Denote the coordinate vectorfields $\p_{x^i} = X_i$. Then compute
\begin{align*}
\p_m \lrangle{\nabla_{X_k} X_1,X_i}_g & = \lrangle{\nabla_{X_m}\nabla_{X_k}X_1,X_i}_g + \lrangle{ \nabla_{X_k}X_1,\nabla_{X_m}X_i}_g \\
& =  \lrangle{\nabla_{X_m}\nabla_{X_1}X_k,X_i}_g + \lrangle{ \nabla_{X_k}X_1,\nabla_{X_m}X_i}_g \\
& = \lrangle{R(X_m,X_1)X_k,X_i}_g + \lrangle{\nabla_{X_1}\nabla_{X_m}X_k,X_i}_g + \lrangle{ \nabla_{X_k}X_1,\nabla_{X_m}X_i}_g \\
& =  \lrangle{R(X_m,X_1)X_k,X_i}_g + \p_1\lrangle{\nabla_{X_m}X_k,X_i}_g \\
& \qquad \qquad \qquad - \lrangle{\nabla_{X_m}X_k,\nabla_{X_1}X_i}_g + \lrangle{ \nabla_{X_k}X_1,\nabla_{X_m}X_i}_g.
\end{align*}
Restricting to ${\varphi_{\text{I}}}$, equations \eqref{covder} and \eqref{christoffder} yield 
\[ \p_m \Gamma_{k1}^i \bigg|_{\varphi_{\text{I}}} = R_{m1ki}\bigg|_{\varphi_{\text{I}}} + \p_1 \Gamma_{mk}^i\bigg|_{\varphi_{\text{I}}}.\]
The second term on the right hand side vanishes due to \eqref{weaknull}.

\end{proof}

From this lemma we immediately see that $\p_m \Gamma_{11}^1 |_{\varphi_{\text{I}}} = R_{m111}|_{\varphi_{\text{I}}} = 0$ from the anti-symmetric property of the Riemann curvature tensor. 
 On the other hand, $\p_m\Gamma_{11}^i|_{\varphi_{\text{I}}} = R_{m11i}|_{\varphi_{\text{I}}}$, which in general does not vanish. We have then proved the following proposition.
\begin{prop}\label{decoupledeqs}
The perturbation equation \eqref{perturbation} decouples into the following system of wave and Klein-Gordon equations for the unknowns $(u^1,\vec{u}\hspace{.05cm})$:
\begin{multline} \label{waveu1}
\Box u^1  = \\
- 2 \sum_{m=2}^d R_{m1j1}(\ell,\vec{0}) u^m \cdot \M (\D u^j,\D \ell)  - \sum_{\substack{|A| = 2 \\ A \neq (1,1)}} \p_A \Gamma_{11}^1(\ell,\vec{0}) \bm u^A \cdot \sigma \\
- \sum_{m=2}^n\p_m\Gamma_{jk}^1(\ell,\vec{0}) u^m \cdot \M(\D u^j,\D u^k) - 2 \sum_{\substack{|A| = 2  \\ A \neq (1,1)}} \p_A \Gamma_{j1}^1(\ell,\vec 0)\bm u^A  \cdot \M(\D u^j, \D \ell) \\
-  \sum_{\substack{|A| = 3 \\ A \neq (1,1,1)}}^n \p_A \Gamma_{11}^1(\ell,\vec{0}) \bm u^A \cdot\sigma + \text{h.o.t.},
\end{multline}

\begin{multline}\label{KGui}
\Box u^i + \sum_{m=2}^nR_{m11i}(\ell,\vec{0}) u^m\sigma = \\
- 2 \sum_{m=2}^d R_{m1ji}(\ell,\vec{0}) u^m \cdot \M(\D u^j,\D \ell)  - \sum_{\substack{|A| = 2 \\ A \neq (1,1)}} \p_A \Gamma_{11}^i(\ell,\vec{0}) \bm u^A \cdot \sigma \\
- \sum_{m=2}^n\p_m\Gamma_{jk}^i(\ell,\vec{0}) u^m \cdot \M(\D u^j,\D u^k) - 2 \sum_{\substack{|A| = 2 \\ A \neq (1,1)}} \p_A \Gamma_{j1}^i(\ell,\vec 0) \bm u^A \cdot \M(\D u^j, \D \ell) \\
-  \sum_{\substack{|A| = 3 \\ A \neq (1,1,1)}} \p_A \Gamma_{11}^i(\ell,\vec{0}) \bm u^A \cdot\sigma + \text{h.o.t.}
\end{multline} 
\end{prop}

\subsection{Reductions when $M$ is a spaceform}
We now suppose that $(M,g)$ is a spaceform with constant sectional curvature $\kappa \neq 0$. In this case the Riemann curvature tensor has the following form:
\begin{equation}\label{curvature}
R_{ijkl} = \kappa(g_{ik}g_{jl} - g_{il}g_{jk}).
\end{equation}
This curvature restriction has the following immediate consequence:
\begin{lem} \label{lem:weaknull2}
Let $m, p \in \{2,\dots,n\}$ and denote $\bullet$ for any element of $\{1,\dots,n\}$. Then, restricted to the geodesic $\varphi_{\text{I}}$,
\begin{equation} \label{weaknull2}
\p^2_{1m} \Gamma_{\bullet 1}^\bullet  = \p^2_{pm} \Gamma_{1 1}^{\bullet} = \p^3_{1\bullet \bullet }\Gamma_{11}^\bullet = 0.
\end{equation}
\end{lem} 
\begin{proof}
Denoting the coordinate vectorfields $\p_{x^i}$ as $X_i$, we have already seen in Lemma \ref{christoffder2} that
\begin{align}
\p_m \lrangle{\nabla_{X_k}X_1,X_i}_g  & = R_{m1ki} + \p_1 \lrangle{\nabla_{X_m}X_k,X_i}_g  - \lrangle{\nabla_{X_m}X_k, \nabla_{X_1}X_i}_g \label{identity} \\
& + \lrangle{\nabla_{X_k}X_1,\nabla_{X_m}X_i}_g\notag
\end{align}
for any $k,i \in \{1,\dots,n\}$. Taking $\p_1$ of both sides shows
\begin{align*}
\p^2_{1m} \lrangle{\nabla_{X_k}X_1,X_i}_g  & = \p_1R_{m1ki} + \p^2_{11} \lrangle{\nabla_{X_m}X_k,X_i}_g - \lrangle{\nabla_{X_1}\nabla_{X_m}X_k, \nabla_{X_1}X_i}_g  \\
& - \lrangle{\nabla_{X_m}X_k, \nabla_{X_1}\nabla_{X_1}X_i}_g + \lrangle{\nabla_{X_1}\nabla_{X_k}X_1,\nabla_{X_m}X_i}_g +  \lrangle{\nabla_{X_k}X_1,\nabla_{X_1}\nabla_{X_m}X_i}_g.
\end{align*}
Restricting this identity on the geodesic proves 
\[ \p^2_{1m} \Gamma_{k 1}^i = \p_1 R_{m1ki} + \p^2_{11}\Gamma_{mk}^i\] 
using Lemma \ref{geotools}. The second term vanishes because of \eqref{weaknull}. The first term vanishes using the spaceform restriction \eqref{curvature} and \eqref{metricder}, proving $\p^2_{1\bullet}\Gamma_{\bullet 1}^\bullet = 0$.

We can instead differentiate \eqref{identity} by $\p_p$ and setting $k = 1$ to deduce
\begin{align*}
\p^2_{pm} \lrangle{\nabla_{X_1}X_1,X_i}_g  & = \p_pR_{m11i} + \p^2_{p1} \lrangle{\nabla_{X_m}X_1,X_i}_g - \lrangle{\nabla_{X_p}\nabla_{X_m}X_1, \nabla_{X_1}X_i}_g  \\
& - \lrangle{\nabla_{X_m}X_1, \nabla_{X_p}\nabla_{X_1}X_i}_g + \lrangle{\nabla_{X_p}\nabla_{X_1}X_1,\nabla_{X_m}X_i}_g +  \lrangle{\nabla_{X_1}X_1,\nabla_{X_p}\nabla_{X_m}X_i}_g.
\end{align*}
Restricting to the geodesic and using Lemma \ref{geotools} similarly proves
\[ \p^2_{pm} \Gamma_{11}^i = \p_p R_{m11i} + \p^2_{p1}\Gamma_{m1}^i.\]
Again the curvature term vanishes using the spaceform restriction \eqref{curvature} and \eqref{metricder}, while the second term vanishes using the already proved $\p^2_{1\bullet}\Gamma_{\bullet 1}^\bullet = 0$ and that regular partial derivatives commute. 

We have show that, for \emph{arbitrary} $\ell$, $\p^2_{\bullet\bullet}\Gamma_{11}^{\bullet}(\ell,\vec 0) = 0$. Arguing as in \eqref{weaknull}, this proves $\p^3_{1\bullet\bullet}\Gamma_{11}^\bullet(\ell,\vec 0) = 0$, as desired.

\end{proof}
This lemma has important ramifications. Firstly, it shows that the only quadratic terms in \eqref{waveu1} and \eqref{KGui} are of the form $\vec{u}\cdot \M(\D \vec{u}, \D \ell)$ and $\vec{u}\cdot \M(\D u^1,\D \ell)$. Secondly, it shows that the \emph{undifferentiated} wave factor $u^1$ is \emph{missing} from the nonlinearities. This null-structure allows our argument to run because the missing resonant terms such as $(u^1)^2$ or $(u^1)^3$ could potentially blow up in finite time due to the lack of the availability of the Morawetz multiplier.

Lemmas \ref{geotools} and \ref{lem:weaknull2}, and Proposition \ref{decoupledeqs} immediately imply that the perturbation equations simplify to 
\begin{multline} \label{waveu1SF}
\Box u^1  = - 2\kappa \sum_{m=2}^n u^m\M(\D u^m,\D \ell)  - 2 \sum_{m,p = 2}^n \p_{mp}^2 \Gamma_{j1}^1(\ell,\vec 0) u^m u^p \cdot \M(\D u^j, \D \ell) \\
 - \sum_{m=2}^n  \p_m\Gamma_{jk}^1(\ell,\vec{0}) u^m \cdot \M(\D u^j,\D u^k)  - \sum_{m,p,q=2}^n \p_{mpq}^3 \Gamma_{11}^1(\ell,\vec{0}) u^m u^p  u^q\cdot \sigma + \text{h.o.t}. ,
\end{multline}
\begin{multline}\label{KGuiSF}
\Box u^i - \kappa u^i \cdot\sigma  = 2\kappa  u^i \cdot \M(\D u^1,\D \ell)  - 2 \sum_{m,p = 2}^n \p_{mp}^2 \Gamma_{j1}^i(\ell,\vec 0) u^m u^p \cdot \M(\D u^j, \D \ell)   \\
-  \sum_{m=2}^n \p_m\Gamma_{jk}^i(\ell,\vec{0}) u^m \cdot \M(\D u^j,\D u^k)  -  \sum_{m,p,q=2}^n\p_{mpq}^3 \Gamma_{11}^i(\ell,\vec{0}) u^m u^p  u^q\cdot \sigma+  \text{h.o.t}.
\end{multline}
\subsubsection{Negatively curved case}
Without loss of generality, in the case of negative sectional curvature we assume $\kappa \equiv -1$. Consequently we demand that the line $\ell$ be timelike ($\sigma < 0$) in order to make the masses of the Klein-Gordon solutions $\vec{u}$ positive. Without loss of  generality, up to a change of coordinates, $\ell \equiv t$. Equations of motions \eqref{waveu1SF} and \eqref{KGuiSF} then reduce to 

\begin{multline} \label{waveu1SFver2}
\Box u^1  = - 2 \sum_{m=2}^n u^m \cdot u_t^m   - 2 \sum_{m,p = 2}^n \p_{mp}^2 \Gamma_{j1}^1(\ell,\vec 0) u^m u^p \cdot \M(\D u^j, \D \ell)\\
 - \sum_{m=2}^n  \p_m\Gamma_{jk}^1(\ell,\vec{0}) u^m \cdot \M(\D u^j,\D u^k)  + \sum_{m,p,q=2}^n \p_{mpq}^3 \Gamma_{11}^1(\ell,\vec{0}) u^m u^p  u^q + \text{h.o.t}. ,
\end{multline}
\begin{multline}\label{KGuiSFver2}
\Box u^i -  u^i =  2  u^i \cdot u^1_t - 2 \sum_{m,p = 2}^n \p_{mp}^2 \Gamma_{j1}^i(\ell,\vec 0) u^m u^p \cdot \M(\D u^j, \D \ell)  \\
-  \sum_{m=2}^n \p_m\Gamma_{jk}^i(\ell,\vec{0}) u^m \cdot \M(\D u^j,\D u^k)  +  \sum_{m,p,q=2}^n\p_{mpq}^3 \Gamma_{11}^i(\ell,\vec{0}) u^m u^p  u^q+  \text{h.o.t}.
\end{multline}

\subsubsection{Positively curved case} \label{subsec:pos}
In the case of positive sectional curvature, we assume $\kappa \equiv +1$ and hence, without loss of generality, we can prescribe $\ell\equiv x^1$. This reduces  \eqref{waveu1SF} and \eqref{KGuiSF} to
\begin{multline} \label{waveu1SFver3}
\Box u^1  = 2 \sum_{m=2}^n u^m \cdot u_{x^1}^m  - 2 \sum_{m,p = 2}^n \p_{mp}^2 \Gamma_{j1}^1(\ell,\vec 0) u^m u^p \cdot \M(\D u^j, \D \ell)\\
 - \sum_{m=2}^n  \p_m\Gamma_{jk}^1(\ell,\vec{0}) u^m \cdot \M(\D u^j,\D u^k)  - \sum_{m,p,q=2}^n \p_{mpq}^3 \Gamma_{11}^1(\ell,\vec{0}) u^m u^p  u^q + \text{h.o.t}. ,
\end{multline}
\begin{multline}\label{KGuiSFver3}
\Box u^i -  u^i =  - 2  u^i \cdot u^1_{x^1}   - 2 \sum_{m,p = 2}^n \p_{mp}^2 \Gamma_{j1}^i(\ell,\vec 0) u^m u^p \cdot \M(\D u^j, \D \ell) \\
-  \sum_{m=2}^n \p_m\Gamma_{jk}^i(\ell,\vec{0}) u^m \cdot \M(\D u^j,\D u^k)  -  \sum_{m,p,q=2}^n\p_{mpq}^3 \Gamma_{11}^i(\ell,\vec{0}) u^m u^p  u^q+  \text{h.o.t}.
\end{multline}
\begin{rmk}
Starting now and for the remainder of the paper we restrict ourselves to the most difficult case of spatial dimension $d = 3$. This is a borderline case in the sense that the linear decay rate for waves $t^{-1}$ barely misses to be integrable. We will overcome this growth by exploiting the weak null-condition present in \eqref{waveu1SFver2} -- \eqref{KGuiSFver3}, namely that resonant wave--wave nonlinearities are \emph{not present}. Instead, we see that the strongest nonlinear interactions are of wave--Klein-Gordon type. Our estimates will close by exploiting the \emph{stronger} linear decay rate of $t^{-3/2}$ for the Klein-Gordon equation. 

Using the integrable decay rate of $t^{(1-d)/2}$ when $d \ge 4$ for linear waves, it is straight forward to show that our results hold for higher dimensions as well. 
\end{rmk}

\section{Basic analytical tools}\label{section:tools}

We will approach the analysis of the system \eqref{waveu1SFver2}--\eqref{KGuiSFver2} using a variant of the vectorfield method adapted to the hyperboloidal foliations. In particular, we will make use of both the Morrey-type global Sobolev inequality developed by Wong in \cite{Wong2017}, which is a refinement of the one in \cite{LeFloch}, and the interpolated GNS-type counterparts developed by the first author and Wong in \cite{GNAWS}. These inequalities and the robust energy method allow us to prove our estimates using only the $T$ multiplier as well as using only the Lorentz boosts as commutator fields.

In this section we develop the notation required for this vectorfield method approach, as well as record the results from \cite{Wong2017, GNAWS} that are needed. We denote by $\Sigma_\tau \subset \R^{1+3}$ the hyperboloid 
\[ \Sigma_\tau \eqdef \{ (t,x) \in \R^{1+3} \ | \ t^2 - |x|^2 = \tau^2, t > 0\}.\] 
It can be parametrized by $\R^3$ via the usual map
\[ (x^1,x^2,x^3) \mapsto ( \sqrt{\tau^2 + |x|^2}, x^1,x^2,x^3) \in \Sigma_\tau \subset \R^{1+3}.\] 
We also denote by 
\[ w_\tau(x) \eqdef \sqrt{\tau^2 + |x|^2}, \qquad x \in \R^3.\]
Of course, $w_\tau \equiv t$ when thinking of $\Sigma_\tau$ as an embedded submanifold of Minkowski space; we use the notation $w_\tau$ because it is useful to work intrinsically on $\Sigma_\tau$. A direct computation shows that the induced volume form on $\Sigma_\tau$ from the Minkowski metric $\M$ on $\R^{1+3}$ is given by
\[ \dvol_{\Sigma_\tau} = \frac{\tau}{w_\tau} \ \D x^1 \wedge \D x^2 \wedge \D x^3.\] 
The commutator fields we will be using are the Lorentz boosts
\[L^i \eqdef t \p_{x^i} + x^i \p_t,\qquad i = 1,2,3\] 
which are Killing vectorfields of Minkowski space. These vectorfields are \emph{tangent} to the hypersurfaces $\Sigma_\tau$ and also span every fibre of the tangent bundle $T\Sigma_\tau$. Since $L^i w_\tau = x^i, \ L^ix^i = w_\tau$, we have that for any string of derivatives
\begin{equation}\label{derweights}
|L^{i_1}\cdots L^{i_k}w_\tau| \le w_\tau, \qquad |L^{i_1}\cdots L^{i_k} \tfrac{x^i}{w_\tau}| \lesssim 1.
\end{equation}
If $\alpha$ is an $m$-tuple with elements drawn from $\{1,2,3\}$ (namely that $\alpha = (\alpha_1,\dots,\alpha_m)$ with $\alpha_i \in \{1,2,3\}$) we denote
\[ L^\alpha u \eqdef L^{\alpha_m}L^{\alpha_{m-1}}\cdots L^{\alpha_1} u.\]
By $|\alpha|$ we refer to its length, namely $m$. 

We define the weighted Lebesgue and Sobolev norms as 
\begin{itemize} %(%(
	\item For $p\in [1,\infty)$ and $\alpha\in\R$, by $\Lebw^p_{\alpha}$ we refer to %]
		\[ \|u\|_{\Lebw^p_\alpha} \eqdef \left( \int w_\tau^\alpha |u|^p \ \dvol_{\Sigma_\tau} \right)^{1/p}. \]
	\item For $p\in [1,\infty)$, $\alpha\in\R$, and $k\in \mathbb{N}$, by $\mr\Sobw^{k,p}_\alpha$ we refer to %]
		\[ \|u\|_{\mr\Sobw^{k,p}_\alpha} \eqdef \sum_{|\beta| = k} \|L^\beta u\|_{\Lebw^p_\alpha}.\]
		The corresponding inhomogeneous version $\Sobw^{k,p}_\alpha$ is 
		\[ \|u\|_{\Sobw^{k,p}_\alpha} \eqdef \sum_{j = 0}^k \|u\|_{\mr\Sobw^{j,p}_\alpha}.\]
\end{itemize}

The main results we need are:
\begin{thm}\emph{(\cite[Theorem 2.18]{Wong2017})}.
Let $l \in \R$ be fixed. For any function $u$ defined on $\Sigma_\tau$, the following uniform estimate holds:
\begin{equation} \label{Morrey-Sobolev}
\tau^{1/2} \| u \cdot w_\tau^{(2+ l )/2}\|_{L^\infty(\Sigma_\tau)} \lesssim  \| u\|_{\Sobw_l^{2,2}}.
\end{equation}
\end{thm}
\begin{prop}\emph{(\cite[Propositions 3.1, 3.7]{GNAWS})}. 
For any function $u$ defined on $\Sigma_\tau$, and for all $r \in [2,6]$, the following estimates hold:
\begin{align}
\tau^{1/2 - 1/r}\| u\|_{\mr \Sobw_{r/2 - 2}^{k,r}} & \lesssim \| u\|_{\mr \Sobw_{-1}^{k,2}}^{\frac{6-r}{2r}} \cdot \| u\|_{\mr \Sobw_{-1}^{k+1,2}}^{\frac{3r-6}{2r}} \label{pqrwave} \\
\tau^{1/2 - 1/r}\| u\|_{\mr \Sobw_{1}^{k,r}} &  \lesssim \| u\|_{\mr \Sobw_{1}^{k,2}}^{\frac{6-r}{2r}} \cdot \| u\|_{\mr \Sobw_{-1}^{k+1,2}}^{\frac{3r-6}{2r}} \label{pqrKG}
\end{align}
\end{prop}
\begin{lem}\emph{(\cite[Lemma 5.1]{Wong2017})}.
Let $d \ge 3$. For any function $u$ defined on $\Sigma_\tau$,
\begin{align}
\| u \|_{\Lebw_{-1}^2} \le \frac{2}{d-2} \| u\|_{\mr \Sobw_{-1}^{1,2}}. \label{Hardy}
\end{align}
\end{lem}

We will denote by 
\begin{align}
 \mathcal{E}_\tau^W[u^1]  &  \eqdef \left(  \int_{\Sigma_\tau}\left( \tau^{-1}\sum_{j=1}^3 |L^ju^1|^2 + \tau (\p_tu^1)^2\right)w_\tau^{-1} \ \dvol_{\Sigma_\tau} \right)^{1/2} \label{Wenegy}, \\
 \mathcal{E}_\tau^{KG}[u^i] &  \eqdef  \left(  \int_{\Sigma_\tau}\left( \tau^{-1}\sum_{j=1}^3 |L^ju^i|^2 + \tau (\p_tu^i)^2 \right)w_\tau^{-1} +  \tau^{-1} |u^i|^2 w_\tau \ \dvol_{\Sigma_\tau} \right)^{1/2 }\label{KGenergy}
 \end{align} 
as the energies of $(u^1,\vec{u}\hspace{.05cm})$.
\begin{rmk}
The Klein-Gordon energy is adapted to Klein-Gordon solutions with mass 1, see subsubsection \ref{subsec:pos}.
\end{rmk}
In the weighted Sobolev notation,
\begin{align*}
 \mathcal{E}_\tau^W[u^1]  & \approx \tau^{-1/2} \|u\|_{\mr\Sobw^{1,2}_{-1}} + \tau^{1/2} \|\partial_t u\|_{\Lebw^2_{-1}}, \\
 \mathcal{E}_\tau^{KG}[u^i] & \approx  \tau^{-1/2} \|u^i\|_{\mr\Sobw^{1,2}_{-1}} + \tau^{1/2} \|\partial_t u^i\|_{\Lebw^2_{-1}} + \tau^{-1/2} \|u^i\|_{\Lebw^2_1} .
 \end{align*}

  These energies satisfy the fundamental energy identity
 \begin{multline} \label{fundamental}
 \mathcal{E}_{\tau_1}^W[u^1]^2 + \sum_{i=1}^n \mathcal{E}_{\tau_1}^{KG}[u^i]^2 \lesssim  \mathcal{E}_{\tau_0}^W[u^1]^2 + \sum_{i=1}^n \mathcal{E}_{\tau_0}^{KG}[u^i]^2 \\
 + \int_{\tau_0}^{\tau_1} \int_{\Sigma_\tau} \Box_\M u^1\p_t u^1 + \langle\Box_\M \vec{u}- \vec{u} , \p_t \vec{u} \hspace{.05cm} \rangle \ \dvol_{\Sigma_\tau} \D \tau,
 \end{multline}
where we wrote
\[\langle \vec \phi, \vec \psi \rangle \eqdef \sum_{i=2}^n \phi^i \psi^i.\]
We use this notation through the rest of this manuscript. We also schematically write
 \[ \mathcal{E}_\tau[\bm u] \eqdef \sqrt{ \mathcal{E}_\tau^W[u^1]^2 + \sum_{i=1}^n \mathcal{E}^{KG}_\tau[u^i]^2}.\] With this the fundamental estimate simplifies to
 \begin{equation}\label{SchematicEnEst}
 \mathcal{E}_{\tau_1}[\bm u]^2 - \mathcal{E}_{\tau_0}[\bm u]^2 \lesssim   \int_{\tau_0}^{\tau_1} \int_{\Sigma_\tau} \Box_\M u^1\p_t u^1 + \langle\Box_\M \vec{u}- \vec{u} , \p_t \vec{u} \hspace{.05cm} \rangle \ \dvol_{\Sigma_\tau} \D \tau.
 \end{equation}

In order to apply \eqref{fundamental} to attain higher derivative estimates of $\bm u$, we commute the system \eqref{waveu1SFver2} -- \eqref{KGuiSFver2} with the Lorentz boosts. It is useful to introduce a notation for higher order energies in order to close the bootstrap assumption in a systematic way. We define
\begin{equation} \label{totalenergy}
\fk{E}_k(\tau) \eqdef \tau^{-1/2} \| \bm u\|_{\Sobw^{k+1,2}_{-1}} + \tau^{1/2} \| \p_t \bm u\|_{\Sobw^{k,2}_{-1}} + \tau^{-1/2} \| \vec{u} \hspace{.05cm} \|_{\Sobw^{k,2}_1}.
\end{equation}
Recall that $\bm u = (u^1,\vec{u}\hspace{.05cm})$ and so $\fk{E}_k(\tau)$ is effectively the total $k$-th order energy of the system.  Indeed, using the commutator algebra properties (specifically those of $[L^i,\partial_t]$) described in \cite[Section 3.2]{AbbWong2019} and the Hardy inequality \eqref{Hardy}, we see that
\[ \fk{E}_k(\tau) \approx \sum_{|\alpha| \le k} \mathcal{E}_\tau[L^\alpha \bm u].\]

\begin{rmk} Hardy's inequality is not available in dimension 2 and so the correct analogue for $\fk{E}_k(\tau)$ on $\R^{1+2}$ is
\[ \fk{E}_k(\tau) \eqdef \tau^{-1/2} \sum_{j=1}^{k+1} \| \bm u\|_{\mr \Sobw_{-1}^{j,2}} + \tau^{1/2} \| \p_t \bm u\|_{\Sobw^{k,2}_{-1}} + \tau^{-1/2} \| \vec{u} \hspace{.05cm} \|_{\Sobw^{k,2}_1}.\]
Note that the Klein-Gordon terms $\vec{u}$ can go all the way down to $k = 0$ because of the mass terms in \eqref{KGenergy}.
\end{rmk}
The following proposition is an immediate consequence of \eqref{Morrey-Sobolev} -- \eqref{Hardy} and the definition of the energy:
\begin{prop} \label{prop:estimates}
For any $x \in \Sigma_\tau$, the following pointwise estimates hold:
\begin{align*}
|\vec u(x)| & \lesssim w_\tau(x)^{-3/2} \fk{E}_2(\tau), \\
|L^i \bm u(x)| + \tau |\p_t\bm u(x)|  & \lesssim w_\tau(x)^{-1/2} \fk{E}_2(\tau). \\
\end{align*}
The following Sobolev estimates hold for any $r \in [2,6]$:
\begin{align*}
\| \vec u \hspace{.05cm}\|_{\mr \Sobw_1^{k,r}(\Sigma_\tau)} & \lesssim \tau^{1/r}\fk{E}_k(\tau), \\
 \|\bm u\|_{\mr \Sobw_{r/2 - 2}^{k+1,r}(\Sigma_\tau)} + \tau \|\p_t \bm u\|_{\mr \Sobw_{r/2-2}^{k,r}(\Sigma_\tau)} & \lesssim \tau^{1/r}   \mathfrak{E}_k(\tau)^{\frac{6-r}{2r}} \cdot  \mathfrak{E}_{k+1}(\tau)^{\frac{3r-6}{2r}} . \label{pqrGNS-waveder-d=3}
\end{align*}
\end{prop}

\section{Global stability in the setting of {\textbf{TL}}}\label{section:negcurv}
In this section we use the estimates recorded in the former in order to prove global existence to the following wave--Klein-Gordon system:
\begin{equation} \label{negsystem}
\begin{aligned}
\Box_\M u^1 &= -2 \langle\vec{u}, \p_t\vec{u}  \hspace{.05cm} \rangle + (\vec{u}\hspace{.05cm})^3 + (\vec u \hspace{.05cm})^2 \p_t \bm u +  \vec{u} \cdot \M(\D \bm u, \D \bm u), \\
\Box_\M u^i - u^i & = 2u^i \p_t u^1 + (\vec{u}\hspace{.05cm})^3 +(\vec u \hspace{.05cm})^2 \p_t \bm u + \vec{u}\cdot \M(\D \bm u, \D \bm u) , \quad i = 2,\dots,n
\end{aligned}
\end{equation}
where here $(\vec{u}\hspace{.05cm})^3,\ (\vec u)^2 \p_t \bm u,$ and $\vec{u} \cdot \M (\D \bm u, \D \bm u)$ is an abuse of notation representing a linear combination of terms of the form 
\begin{equation} \label{cubicterms}
\begin{aligned}
u^m u^p u^q, & & m,p,q \in \{2,\dots,n\}, & \\
u^m u^p \p_t u^j, && m,p \in \{2,\dots,n\}, && j \in \{1,\dots,n\}, \\
u^m \M(\D u^j,\D u^k), & & m \in \{2,\dots,n\}, &&  j,k \in \{1,\dots,n\}.
\end{aligned}
\end{equation}
For our convenience, we will prescribe initial data at $t= 2$:
\[ \bm u(2,x) = \bm \phi_0(x),\qquad \p_t \bm u(2,x) = \bm \varphi_0(x).\]

Even though this system is a simplification of \eqref{waveu1SFver2}--\eqref{KGuiSFver2}, it captures all of the analytical difficulties and extending the results to the full equations of motion is merely a matter of bookkeeping. Indeed, as the coefficients of \eqref{cubicterms} in the full system are of the form $\p \Gamma(\ell,0)$, they can be regarded as universal constants as a consequence of Lemma \ref{geotools} coupled with the fact that we consider the manifold $(M,g)$ fixed with constant curvature. Moreover, the higher ordered terms in \eqref{waveu1SFver2}--\eqref{KGuiSFver2} are 
\[\left( |\vec{u}\hspace{.05cm}|^4 + |\vec{u}\hspace{.05cm}|^3|\p_t \bm u| + |\vec{u}\hspace{.05cm}|^2|\M(\D \bm u, \D \bm u)|\right) f(\bm u,\p \bm u),\]
where $f : \R^{n(d+2)} \to \R$ is a smooth function depending on $M$. The standard argument, using the energy method, for either the stability problem or the local existence problem for nonlinear waves, handles the nonlinearities with the general prescription of ``putting the highest order derivative factor in $L^2$ and the remainder in $L^\infty$.'' As the $L^\infty$ estimates we will be using are the pointwise bounds from Proposition \ref{prop:estimates}, we see that higher order nonlinearities lead to \emph{more} available decay, and hence add no difficulties when improving the bootstrap assumptions.

Our main theorem asserts that a geodesic wave map affinely parametrized by a timelike \emph{linear} free wave is stable under small (in an appropriate Sobolev norm) perturbations, and that the perturbed solution stays within a small tubular neighborhood of the background geodesic. 
\begin{thm}\label{maintheorem1}
For any $\gamma < 1/2$, there exists some $\epsilon_0$ (which depends only on $\gamma$) such that whenever $\bm \phi_0, \bm \varphi_0$ are compactly supported in the ball of radius 1 centered at the origin satisfying
\[\|\bm \phi_0\|_{H^3} + \| \bm \varphi_0\|_{H^2} < \epsilon_0,\]
there exists a unique solution $\bm u = (u^1,\vec u \hspace{.05cm})$ to \eqref{negsystem} that exists for all time $t \ge 2$. Furthermore, we have the following uniform estimates:
\begin{align*}
|u^1| + \sum_{i=1}^3 |L^i \bm u \hspace{.05cm}| & \lesssim \tau^{\gamma}t^{-1/2} \\
|\vec u\hspace{.05cm}| &  \lesssim \tau^{\gamma} t^{-3/2} \\
|\p_t \bm u|& \lesssim \tau^{-1 + \gamma} t^{-1/2}.
\end{align*}
\end{thm}

By standard local existence theory we can assume that for sufficiently small initial data, the solution $\bm u$ of \eqref{negsystem} exists up to $\Sigma_2$. The breakdown criterion for wave and Klein-Gordon equations imply that so long as we can show that $|\vec{u} \hspace{.05cm} |, \ |L\bm u|, \ |\p_t \bm u|$ remain bounded on $\Sigma_\tau$ for all $\tau > 2$, we can guarantee global existence of solutions. Proposition \ref{prop:estimates} implies that a sufficient condition for global existence are a priori estimates on the second order energies. The general approach is that of a bootstrap argument:
\begin{enumerate}
	\item We will assume that, up to time $\tau_{\text{max}} > 2$, that the energies $\fk{E}_k(\tau)$ of the solution $\bm u$ and its derivatives $L^\alpha\bm u$ verify certain bounds. 
	\item Using Proposition \ref{prop:estimates}, this gives $L^\infty$ bounds on $\bm u$, and its derivatives of the form $L^\alpha \bm u$ and $\p_t L^\alpha\bm u$. 
	\item We can then estimate the nonlinearity using these $L^\infty$ estimates, which we then feed back into the energy inequality \eqref{SchematicEnEst} to get an \emph{updated} control on $\fk{E}_k(\tau)$ for all $\tau \in [2,\tau_{\text{max}}]$. 
	\item Finally, show for sufficiently small initial data sizes, that the updated control \emph{improves} the original control, whereupon by the method of continuity the original bounds on $\fk{E}_k(\tau)$ must hold for all $\tau \geq 2$, implying the desired global existence. 
\end{enumerate}
Since the Lorentz boosts $L^i$ commute with the d'Alembertian $[L^i,\Box_\M] = 0$, after applying \eqref{SchematicEnEst} to $L^\alpha \bm u = (L^\alpha u^1, L^\alpha\vec{u} \hspace{.05cm})$ we see that we need to estimate the integrals
\begin{equation} \label{nonlinearities}
\int_{\Sigma_\tau} L^\alpha( \Box_\M u^1) \p_t L^\alpha u^1 + \langle L^\alpha(\Box_\M \vec{u} - \vec{u}\hspace{.05cm}), \p_t L^\alpha \vec{u}\hspace{.05cm}\rangle \ \dvol_{\Sigma_\tau}
\end{equation}
for all tuples $\alpha$ with elements drawn from $\{1,2,3\}$ and length $\le 2$. 

From the structure of \eqref{negsystem}, when $|\alpha| = 0$ we see a \emph{complete} cancellation of the quadratic terms in \eqref{nonlinearities}:
\begin{equation}\label{eq:quadcancel}
-2 \langle \vec u, \p_t \vec u \hspace{.05cm}\rangle \p_t u^1 + 2 \langle \p_t u^1 \vec u, \p_t \vec u\hspace{.05cm}\rangle = 0.
\end{equation}
Although this cancellation is unique to the case of $|\alpha| = 0$, for $|\alpha| = 1, 2$ we do see a cancellation of all of the \emph{top order} derivative quadratic terms. We see for any tuple $\alpha$ with elements drawn from $\{1,2,3\}$
\begin{multline}\label{quadratic1}
\left| L^\alpha \big(-2 \langle \vec u,\p_t \vec u\hspace{.05cm}\rangle\big) \p_t L^\alpha u^1 + \left\langle L^\alpha (2\p_t u^1 \vec{u}\hspace{.05cm}),\p_t L^\alpha \vec{u} \right\rangle  \right|  \lesssim \\
\left| - \langle \vec{u},L^\alpha\p_t \vec u\hspace{.05cm}\rangle \p_t L^\alpha u^1 +  L^\alpha \p_t u^1 \langle\vec{u},\p_t L^\alpha \vec{u}\hspace{.05cm}\rangle \right| \\
+ \sum_{\substack{|\beta| + |\gamma| \le |\alpha| \\ |\beta| \neq |\alpha|}} \left|  \langle L^\gamma \vec{u} ,L^\beta \p_t \vec{u} \hspace{.05cm}\rangle  \p_t L^\alpha u^1 + L^\beta \p_t u^1\langle L^\gamma \vec{u},\p_t L^\alpha \vec{u}\hspace{.05cm}\rangle \right|.
\end{multline}
Using the commutator algebra properties 
\begin{equation}\label{commute}
[L^i,\p_t ] = - \p_{x^i}  = - \frac{1}{t} L^i + \frac{x^i}{t} \p_t,
\end{equation} 
we see a cancellation of the top order terms in the first term on the right hand side of \eqref{quadratic1}. Consequently, the quadratic terms of \eqref{nonlinearities} can be estimated schematically as 
\begin{equation} \label{quadratic2}
\lesssim  \int_{\Sigma_\tau}\sum_{\substack{|\beta| + |\gamma| \le |\alpha| \\ |\beta| \neq \alpha}} |L^\gamma \vec{u} \p_t L^\beta \bm u \p_t L^\alpha \bm u| + \sum_{|\beta| + |\gamma| \le |\alpha|} w_\tau^{-1} |L^\gamma \vec{u} L^\beta \bm u \p_t L^\alpha \bm u| \ \dvol_{\Sigma_\tau}
\end{equation}
where we repeatedly used \eqref{commute} and \eqref{derweights}. We can now estimate the quadratic terms of \eqref{nonlinearities}. 
\begin{prop}[Quadratic energy estimates] \label{quadenergy}
Let $\alpha \neq 0$ be an $m$-tuple\footnote{The case $m = 0$ does not need to be controlled due to \eqref{eq:quadcancel}.} with elements drawn from $\{1,2,3\}$. Then 
\begin{multline*}
\int_{\Sigma_\tau} \left| L^\alpha \big(-2 \langle \vec u,\p_t \vec u\hspace{.05cm}\rangle\big) \p_t L^\alpha u^1 + \left\langle L^\alpha (2\p_t u^1 \vec{u}\hspace{.05cm}),\p_t L^\alpha \vec{u} \right\rangle  \right| \dvol_{\Sigma_\tau}  \lesssim \\
\begin{cases} 
\tau^{-3/2} \fk{E}_1^2 \cdot \fk{E}_2& m = 1, \\
\tau^{-3/2} \fk{E}_2^3 + \tau^{-1} \fk{E}_1\cdot \fk{E}_2^2 + \tau^{-1} \fk{E}_1^2 \cdot \fk{E}_2 & m = 2.
\end{cases}
\end{multline*}
\end{prop}
\begin{proof}
Throughout this proof we use the simple inequality $w_\tau^{-1} \le \tau^{-1}$. We prove the estimate for the case $m = 1$ first. In this case the top ordered derivative terms of \eqref{quadratic2} that we need to estimate are of the form
\[ \int_{\Sigma_\tau}| L \vec{u} \p_t \bm u \p_t L \bm u| + w_\tau^{-1} |\vec{u} L \bm u \p_t L \bm u| \ \dvol_{\Sigma_\tau}\]
(the estimates for the lower ordered terms will of course be controlled by the top ones). Here it is understood that $L$ can be any of the boosts $L^i$. For the first term, we estimate the $\p_t \bm u$ factor by the pointwise estimates in Proposition \ref{prop:estimates} and H\"older's inequality on the rest of them 
\begin{align*}
\int_{\Sigma_\tau} |L\vec{u} \p_t \bm u \p_t L \bm u | \ \dvol_{\Sigma_\tau} & \lesssim \tau^{-3/2} \fk{E}_2(\tau) \int_{\Sigma_\tau} |L\vec{u}\ w_\tau^{1/2}| \cdot|\p_t L \bm u\ w_\tau^{-1/2}| \ \dvol_{\Sigma_\tau} \\
& \le \tau^{-3/2}\fk{E}_2(\tau)\cdot (\fk{E}_1(\tau))^2.
\end{align*}
For the second term, we control the $\vec{u}$ factor by the pointwise estimates and use H\"older's inequality on the rest
\begin{align*}
\int_{\Sigma_\tau}  w_\tau^{-1} |\vec{u} L \bm u \p_t L \bm u| \ \dvol_{\Sigma_\tau} & \lesssim \tau^{-3/2} \fk{E}_2(\tau)  \int_{\Sigma_\tau} |L\bm u\ w_\tau^{-1/2}| \cdot|\p_t L \bm u\ w_\tau^{-1/2}| \ \dvol \\
& \le \tau^{-3/2}\fk{E}_2(\tau)\cdot (\fk{E}_1(\tau))^2.
\end{align*}
This concludes the proof for $m = 1$.

For $m = 2$, the terms from \eqref{quadratic2} are 
\begin{multline*}
\int_{\Sigma_\tau} |LL \vec{u} \p_t \bm u \p_t LL\bm u|  + w_\tau^{-1} |\vec{u} LL \bm u \p_tLL\bm u| \\
+ |L \vec u \p_t L \bm u \p_t LL \bm u| + w_\tau^{-1}| L \vec u L \bm u \p_t LL \bm u|  \ \dvol_{\Sigma_\tau}.
\end{multline*}
Again, the estimates for all lower ordered terms can be controlled by the estimates of these. Here it is understood that $LL$ is any arbitrary second order tangential derivative $L^iL^j$. The first two terms are bounded by 
\[ \tau^{-3/2} (\fk{E}_2(\tau))^3\] using the same techniques as $m = 1$ (estimating the lowest ordered terms in $L^\infty$ and the rest by the energies after using H\"older's). The other two terms cannot be treated with the same techniques. Even though $|L \vec{u}\hspace{.05cm}| + |L \bm u|$ can be bounded by $w_\tau^{-1/2}\fk{E}_2(\tau)$, this decay is too weak to improve the bootstrap assumptions that we will make. On the other hand, we can get stronger decay for the third term above by estimating $|L \vec{u}\hspace{.05cm}| \le w_\tau^{-3/2} \fk{E}_3(\tau)$. This is not helpful to us because $\fk{E}_3(\tau)$ requires square integrability of \emph{four} derivatives of $\bm u$ (recall that we want to solve the Cauchy problem for \eqref{negsystem} using data in $H^3$). 

We instead appeal to the interpolated Sobolev estimates in Proposition \ref{prop:estimates} with $r = 3, 6$ to control the third and fourth terms above. We see
\begin{align}
\int_{\Sigma_\tau} |L \vec u \p_t L \bm u \p_t LL \bm u| \ \dvol_{\Sigma_\tau} & = \int_{\Sigma_\tau} |L \vec{u} \ w_\tau^{1/3}| \cdot | \p_t L \bm u\ w_\tau^{1/6}| \cdot | \p_t LL \bm u\ w_\tau^{-1/2}| \ \dvol_{\Sigma_\tau} \notag \\
& \le \| \vec u \hspace{.05cm} \|_{\mr \Sobw_1^{1,3}} \| \p_t \bm u \|_{\mr \Sobw_1^{1,6}} \| \p_t \bm u \|_{\mr \Sobw^{2,2}_{-1}} \notag \\
& \le\tau^{-1} \fk{E}_1(\tau) \cdot (\fk{E}_2(\tau))^2.\label{borderline}
\end{align}
Similarly, we see 
\begin{align}
\int_{\Sigma_\tau}  w_\tau^{-1}| L \vec u L \bm u \p_t LL \bm u|  \ \dvol_{\Sigma_\tau} & \le \tau^{-1}  \int_{\Sigma_\tau} |L \vec{u} \ w_\tau^{1/3}| \cdot | L \bm u\ w_\tau^{1/6}| \cdot | \p_t LL \bm u\ w_\tau^{-1/2}| \ \dvol_{\Sigma_\tau}\notag \\
& \le \tau^{-1}\| \vec u \hspace{.05cm} \|_{\mr \Sobw_1^{1,3}} \| \bm u \|_{\mr \Sobw_1^{1,6}} \| \p_t \bm u \|_{\mr \Sobw^{2,2}_{-1}} \notag\\
& \le \tau^{-1}( \fk{E}_1(\tau))^2 \cdot (\fk{E}_2(\tau)). \label{nonborderline}
\end{align}
This concludes the proof of the proposition.
\end{proof} 
\begin{rmk}
The expression on the right hand side of \eqref{nonborderline} would allow us to close our energy estimates with only a log loss, see Proposition \ref{prop:bootstrap}. The borderline terms that we need to deal with are in fact in \eqref{borderline}.
\end{rmk}
Estimating the cubic terms in \eqref{negsystem} we identify the integrals that we have to estimate are 
\[ \int_{\Sigma_\tau} \big(L^\beta \vec u \ L^\gamma \vec u\ L^\sigma \vec u + L^\beta \vec u \cdot \M(\D L^\gamma \bm u, \D L^\sigma \bm u) + L^\beta \vec u \ L^\gamma \vec u \ L^\sigma \p_t \bm u \big)\cdot \p_t L^\alpha \bm u \dvol_{\Sigma_\tau}\]
for $|\beta| + |\gamma| + |\sigma| = |\alpha|$. Here we implicitly used that vectorfields act on scalars by Lie differentiation, that $\M$ is invariant under the Lorentz boosts $L^i$, and that exterior differentiation commutes with Lie differentiation. 
\begin{prop}[Cubic energy estimates] \label{cubicenergy} Let $\alpha$ be an $m$-tuple with elements drawn from $\{1,2,3\}$. Then 
\begin{multline*}
\int_{\Sigma_\tau} \left| L^\alpha( (\vec u \hspace{.05cm}))^3 + L^\alpha ( (\vec u \hspace{.05cm})^2 \p_t \bm u) + L^\alpha( \vec{u} \cdot \M(\D \bm u, \D \bm u))\right|\cdot |\p_t L^\alpha \bm u| \ \dvol_{\Sigma_\tau} \lesssim \\
\begin{cases}
\tau^{-3} \fk{E}_m^2\cdot \fk{E}_2^2 & m = 0,1 \\
\tau^{-3}\fk E_2^4 + \tau^{-2} \fk{E}_0^{1/4}\cdot \fk{E}_1\cdot\fk{E}_2^{11/4} + \tau^{-3/2} \fk{E}_1^2 \cdot \fk{E}_2^2 + \tau^{-5/2} \fk{E}_1\cdot\fk{E}_2^3& m = 2
\end{cases}
\end{multline*}
\end{prop}
\begin{proof}
Let us treat the terms with $(\vec{u}\hspace{.05cm})^3$ first. When $m = 0$, we control two of the factors by the pointwise estimates:
\begin{align*}
\int_{\Sigma_\tau} |(\vec u\hspace{.05cm})^3\cdot\p_t \bm u| \ \dvol_{\Sigma_\tau}&  \lesssim \tau^{-3} \fk{E}_2(\tau)^2 \int_{\Sigma_\tau} |\vec u \ w_\tau^{1/2}| \cdot |\p_t \bm u \ w_\tau^{-1/2}| \ \dvol_{\Sigma_\tau} \\
& \lesssim \tau^{-3} \fk{E}_0(\tau)^2\cdot \fk{E}_2(\tau)^2,
\end{align*}
as desired. For $m = 1$, the same proof follows by controlling the two factors that are \emph{not} differentiated by the pointwise estimates (note that the density is $| L\vec u\hspace{.05cm} (\vec u \hspace{.05cm})^2 \p_tL \bm u|$). When $m = 2$, this can again be used to bound the terms of the form 
\[\int_{\Sigma_\tau} |LL \vec u\cdot (\vec u\hspace{.05cm})^2\cdot\p_t LL\bm u| \ \dvol_{\Sigma_\tau}   \lesssim \tau^{-3}\fk{E}_2(\tau)^4.\] 
For the other cases, we couple the pointwise estimates \emph{and} the interpolated GNS estimates of Proposition \ref{prop:estimates} to find
\begin{align*}
\int_{\Sigma_\tau} |(L \vec u \hspace{.05cm})^2\cdot \vec u \cdot \p_t LL \bm u| \ \dvol_{\Sigma_\tau} & \lesssim \tau^{-3/2} \fk{E}_2  \int_{\Sigma_\tau} |L \vec{u} \ w_\tau^{1/3}| \cdot | L \vec u\ w_\tau^{1/6}| \cdot | \p_t LL \bm u\ w_\tau^{-1/2}| \ \dvol_{\Sigma_\tau} \\
& \lesssim \tau^{-3/2} \fk{E}_1(\tau)^2 \cdot \fk{E}_2(\tau)^2.
\end{align*}

Next we control the $(\vec u \hspace{.05cm})^2 \p_t \bm u$ terms. For $m = 0$, we control one $\p_t \bm u$ and one Klein-Gordon term by the energy:
\begin{align*}
\int_{\Sigma_\tau} |(\vec u\hspace{.05cm})^2 \p_t \bm u \cdot \p_t \bm u| \ \dvol_{\Sigma_\tau} & \lesssim \tau^{-3} \fk{E}_2 \int_{\Sigma_\tau} |\vec u\ w_\tau^{1/2}| \cdot |\p_t \bm u \ w_\tau^{-1/2}| \ \dvol_{\Sigma_\tau} \\
& \lesssim \tau^{-3} \fk{E}_0(\tau)^2 \cdot \fk{E}_2(\tau)^2.
\end{align*}
For $m = 1$, the same technique is used to bound
\[ \int_{\Sigma_\tau} | \vec u\ L\vec u\ \p_t u \cdot \p_t L u| \dvol_{\Sigma_\tau} \lesssim \tau^{-3} \fk{E}_1(\tau)^2 \fk{E}_2(\tau)^2.\]
When the derivative hits the $\p_t \bm u$ factor we sacrifice some of the decay given by the Klein-Gordon terms\footnote{Of course, there are lower-ordered terms which appear as a consequence of commuting the derivative: $L^i\p_t \bm u = \p_t L^i \bm u - w_\tau^{-1} L^i \bm u + \tfrac{x^i}{w_\tau} \p_t \bm u $. One can check that the energies of these commuted terms are bounded by $\tau^{-3}\fk{E}_0\fk{E}_1\fk{E}_2^2$. We drop these lower ordered energies because will of course be controlled by $\tau^{-3} \fk{E}_1^2 \fk{E}_2^2$.}:
\begin{align*}
\int_{\Sigma_\tau} | (\vec u\hspace{.05cm})^2 \p_t L\bm u \cdot \p_t L \bm u| \ \dvol_{\Sigma_\tau} & \lesssim \tau^{-2} \fk{E}_2^2 \int_{\Sigma_\tau} | \p_t L\bm u \ w_\tau^{-1/2}| \cdot| \p_t L \bm u \ w_\tau^{-1/2}| \ \dvol_{\Sigma_\tau} \\
& \lesssim \tau^{-2} \fk{E}_2(\tau)^2\cdot \| \p_t u \|_{\mr \Sobw_{-1}^{1,2}}^2 \\
& \lesssim \tau^{-3} \fk{E}_1(\tau)^2 \cdot \fk{E}_2(\tau)^2.
\end{align*}
For $m = 2$ the densities we need to estimate are $\p_t LL\bm u$ multiplied by \footnote{Again, there are lower ordered terms that rise from commuting $L^i$ and $\p_t$. We drop these energies because one can check that they will all be controlled by the energies of $(\vec u\hspace{.05cm})^2 \ \p_t LL \bm u$.}
\begin{equation}\label{eq:densities}
 \vec u \ LL \vec u  \ \p_t \bm u, \qquad (L \vec u \hspace{.05cm})^2 \p_t \bm u, \qquad \vec u \ L \vec u \ \p_t L \vec u,\qquad (\vec u\hspace{.05cm})^2 \p_t LL \bm u.
\end{equation}
For the first density we estimate the undifferentiated terms by the energies as we did for $m = 0$ and $m = 1$ to see
\[ \int_{\Sigma_\tau} | \vec u \ LL \vec u  \ \p_t \bm u \cdot \p_t LL \bm u| \dvol_{\Sigma_\tau} \lesssim \tau^{-3} \fk{E}_2(\tau)^4.\]
The second density of \eqref{eq:densities} is treated with the interpolation inequalities after using the pointwise estimate to control $\p_t \bm u$:
\begin{align*}
\int_{\Sigma_\tau} |(L \vec u \hspace{.05cm})^2 \p_t \bm u \cdot \p_t LL \bm u| \ \dvol_{\Sigma_\tau} & \lesssim \tau^{-3/2} \fk{E}_2 \int_{\Sigma_\tau} | L \vec u \ w_\tau^{1/3}|\cdot| L\vec u \ w_\tau^{1/6}| \cdot |\p_t LL \bm u \ w_\tau^{-1/2} | \ \dvol_{\Sigma_\tau} \\
& \lesssim\tau^{-3/2} \fk{E}_2 \| \vec u\hspace{.05cm} \|_{\mr \Sobw_1^{1,3}} \| \vec u\hspace{.05cm} \|_{\mr \Sobw_{1}^{1,6}} \| \p_t \bm u\|_{\mr \Sobw_{-1}^{2,2}} \\
& \lesssim \tau^{-3/2} \fk{E}_1(\tau)^2 \fk{E}_2(\tau)^2.
\end{align*}
The third density is treated similarly:
\begin{align*}
\int_{\Sigma_\tau} |\vec u \ L \vec u \ \p_t L \vec u  \cdot \p_t LL \bm u| \ \dvol_{\Sigma_\tau} & \lesssim \tau^{-3/2} \fk{E}_2 \int_{\Sigma_\tau} | L \vec u \ w_\tau^{1/3}|\cdot| \p_t L\bm u \ w_\tau^{1/6}| \cdot |\p_t LL \bm u \ w_\tau^{-1/2} | \ \dvol_{\Sigma_\tau} \\
& \lesssim\tau^{-3/2} \fk{E}_2(\tau) \cdot \| \vec u\hspace{.05cm} \|_{\mr \Sobw_1^{1,3}}\cdot \| \p_t \bm  u\hspace{.05cm}  \|_{\mr \Sobw_{1}^{1,6}} \cdot \| \p_t \bm u\|_{\mr \Sobw_{-1}^{2,2}} \\
& \lesssim \tau^{-5/2} \fk{E}_1(\tau) \fk{E}_2(\tau)^3.
\end{align*}
The last case of \eqref{eq:densities} is treated by controlling the two Klein-Gordon factors by the pointwise estimates
\begin{align*}
\int_{\Sigma_\tau} |(\vec u\hspace{.05cm})^2 \p_t LL \bm u \cdot \p_t LL \bm u| \dvol_{\Sigma_\tau} & \lesssim \tau^{-2} \fk{E}_2^2 \int_{\Sigma_\tau} |\p_t LL \bm u \ w_\tau^{-1/2}| \cdot |\p_t LL \bm u \ w_\tau^{-1/2} | \ \dvol_{\Sigma_\tau} \\
& \lesssim \tau^{-2} \fk{E}_2(\tau)^2 \cdot \| \p_t \bm u \|_{\mr \Sobw_{-1}^{2,2}}^2 \\
& \lesssim \tau^{-3} \fk{E}_2(\tau)^4.
\end{align*}

We finally treat the $\vec u \cdot \M(\D \bm u, \D \bm u)$ terms. Note firstly that the second equality in \eqref{commute} implies the estimate
\[ |\M(\D \psi_1,\D\psi_2) | \le \frac{1}{\tau^2} | L\psi_1| \cdot | L\psi_2| + |\p_t\psi_1 | \cdot |\p_t \psi_2|\] 
for any scalars $\psi_1, \ \psi_2$. For $m = 0$ the pointwise estimates imply
\begin{align*}
\int_{\Sigma_\tau} |\vec u \cdot \M(\D \bm u, \D \bm u) \cdot \p_t \bm u| \dvol_{\Sigma_\tau} & \lesssim \tau^{-3} \fk{E}_2(\tau)^2 \int_{\Sigma_\tau} |\vec u \ w_\tau^{1/2}| \cdot |\p_t \bm u \ w_\tau^{-1/2}| \ \dvol_{\Sigma_\tau} \\
& \lesssim \tau^{-3} \fk{E}_0(\tau)^2 \fk{E}_2(\tau)^2.
\end{align*}
For $m = 1$ we similarly see
\[ \int_{\Sigma_\tau} |L\vec u \cdot \M(\D \bm u, \D \bm u) \cdot \p_t L \bm u| \dvol_{\Sigma_\tau}  \lesssim \tau^{-3} \fk{E}_1(\tau)^2 \fk{E}_2(\tau)^2.\]
When the derivative hits the null form factor the density is $\vec u \cdot \M(\D L u,\D u) \cdot \p_tL \bm u$. We can then use the improved Klein-Gordon decay $|\vec u \hspace{.05cm} | \lesssim w_\tau^{-3/2} \fk{E}_2$ to estimate 
\begin{align*}
\int_{\Sigma_\tau} |\vec u \cdot \M(\D L \bm u, \D \bm u) \cdot \p_t L \bm u| \dvol & \lesssim \tau^{-1} \fk{E}_2 \int_{\Sigma_\tau}\left( \tau^{-2} |LL \bm u| \cdot |L \bm u| + |\p_t L \bm u | \cdot| \p_t \bm u| \right)\frac{ |\p_t L \bm u|}{w_\tau^{1/2}} \ \dvol \\
& \lesssim \tau^{-2} \fk{E}_2^2 \int_{\Sigma_\tau}\left(  \frac{|LL \bm u|}{\tau w_\tau^{1/2}} + \frac{|\p_t L \bm u|}{w_\tau^{1/2}} \right) \frac{ |\p_t L \bm u|}{w_\tau^{1/2}} \ \dvol \\
& \lesssim \tau^{-3} \fk{E}_1^2 \fk{E}_2^2.
\end{align*}
Replicating the previous estimates, when $m = 2$, 
\[ \int_{\Sigma_\tau} \left| LL\vec u \cdot \M(\D \bm u, \D \bm u)+ \vec{u} \cdot \M(\D LL\bm u, \D \bm u)\right| \cdot| \p_tLL \bm u| \dvol \lesssim \tau^{-3} \fk{E}_2^4.\]
The remaining term 
\[ \int_{\Sigma_\tau} |L\vec u \cdot \M(\D L \bm u, \D \bm u) \p_t LL \bm u| \ \dvol_{\Sigma_\tau}\] can't be treated in the same way because the improved decay from the Klein-Gordon term comes at a loss of one derivative: $|L\vec u \hspace{.05cm} | \lesssim w_\tau^{-3/2} \fk{E}_3$. We must then rely on the weaker estimate $|L\vec u \hspace{.05cm} | \lesssim w_\tau^{-1/2} \fk{E}_2$ and remedy this loss with the interpolated GNS estimates from Proposition \ref{prop:estimates} with $r = 4$:
\begin{align*}
\int_{\Sigma_\tau} |L\vec u \cdot \M(\D L \bm u, \D \bm u) \cdot \p_t L \bm u| \dvol & \lesssim \fk{E}_2 \int_{\Sigma_\tau} \left( \tau^{-2} |LL\bm u| \cdot |L\bm u| + |\p_t L\bm u| \cdot| \p_t \bm u|\right) \frac{\p_t LL\bm u}{w_\tau^{1/2}}\ \dvol \\
& \lesssim \tau^{-2} \fk{E}_0^{1/4}\cdot \fk{E}_1^{3/4}\cdot \fk{E}_1^{1/4}\cdot \fk{E}_2^{3/4}\cdot \fk{E}_2.
\end{align*}
and the proposition follows.
\end{proof}
Using \eqref{totalenergy}, we have as an immediate corollary of Propositions \ref{quadenergy} and \ref{cubicenergy} the following \emph{a priori} estimates:
\begin{cor} \label{cor-apri-est}
\begin{align}
\fk{E}_0(\tau_1)^2 - \fk{E}_0(\tau_0)^2 & \lesssim \int_{\tau_0}^{\tau_1} \tau^{-3} \fk{E}_0^2 \fk{E}_2^2 \ \D \tau \label{apriori0}\\
\fk{E}_1(\tau_1)^2 - \fk{E}_1(\tau_0)^2 & \lesssim \int_{\tau_0}^{\tau_1} \tau^{-3/2} \fk{E}_1^2 \fk{E}_2 + \tau^{-3} \fk{E}_1^2 \fk{E}_2^2 \ \D \tau \label{apriori1} \\
\fk{E}_2(\tau_1)^2 - \fk{E}_2(\tau_0)^2 & \lesssim \int_{\tau_0}^{\tau_1} \tau^{-1} \fk{E}_1 \fk{E}_2\big(\fk{E}_1+\fk{E}_2\big) + \tau^{-3/2} \fk{E}_2^2\left(\fk{E}_2 + \fk{E}_1\fk{E}_2+  \fk{E}_1^2 \right) \label{apriori2} \\
& \qquad \qquad \qquad    + \tau^{-3} \fk E_2^4 + \tau^{-2} \fk{E}_0^{1/4}\cdot \fk{E}_1\cdot\fk{E}_2^{11/4} \ \D \tau.
\end{align}
\end{cor}
\noindent These estimates imply the following bootstrap estimate.
\begin{prop}\label{prop:bootstrap}
Assume that the initial data satisfy 
\begin{equation}
\fk{E}_2(2) \le \epsilon \label{intialdata}
\end{equation}
and that for some $\tau_{\text{max}} > 2$ the bootstrap assumptions
\begin{equation}\label{auxboot}
\begin{cases}
\fk{E}_0(\tau) \le \delta \\
\fk{E}_1(\tau) \le \delta \\
\fk{E}_2(\tau) \le \delta \tau^\gamma
\end{cases}
\end{equation}
hold for all $\tau \in [2,\tau_{\text{max}}]$ and some $\delta < 1, \ \gamma < \frac{1}{2}$. Then there exists a constant $C$ depending only on $\gamma$ such that the improved estimates 
\begin{equation}\label{improvedboot}
\begin{cases}
\fk{E}_0(\tau) \le \epsilon + C \delta^{3/2} \\
\fk{E}_1(\tau) \le \epsilon + C \delta^{3/2} \\
\fk{E}_2(\tau) \le \epsilon + C \delta^{3/2} \tau^\gamma
\end{cases}
\end{equation}
hold for all $\tau \in [2,\tau_{\text{max}}]$.
\end{prop}
\begin{proof}
Improving the estimate for $\fk{E}_0$ follows from \eqref{apriori0} after noting that 
\[\int_2^\tau \sigma^{-3} \fk{E}_0(\sigma)^2 \cdot \fk{E}_2(\sigma)^2 \ \D \sigma \le \delta^4 \int_2^\tau \sigma^{-3 + 2\gamma} \ \D \sigma \le \delta^4 \int_2^\infty \sigma^{-3 + 2\gamma} \ \D \sigma \le C \delta^3.\]%(
Similarly, the estimate for $\fk{E}_1$ follows from \eqref{apriori1} because $\sigma^{-3/2 + \gamma}$ is integrable for $\sigma \in [2,\infty)$ provided that $\gamma < 1/2$.  

We begin to improve the bootstrap $\fk{E}_2$ by controlling the first two terms in the right hand side of \eqref{apriori2}, which are bounded by 
\[   \delta^3\int_2^\tau \sigma^{-1+2\gamma} \ \D \sigma \le C \delta^3 \tau^{2\gamma}.\]
The rest of the terms are all bounded by 
\[ \delta^3 \int_2^\infty  \sigma^{-3/2 + 3\gamma} \ \D \sigma \le C \delta^3\tau^{2\gamma},\] 
provided that $\gamma < 1/2$. We now consider $\gamma$ fixed once and for all.
\end{proof}

As a consequence of the improved estimates, if we choose $\delta \le (4C)^{-1/2}$ and then $\epsilon < \delta/4$, we we conclude
\[ \begin{cases}
\fk{E}_0(\tau) \le \frac{1}{2} \delta \\
\fk{E}_1(\tau) \le\frac{1}{2} \delta\\
\fk{E}_2(\tau) \le \frac{1}{2}  \delta \tau^\gamma
\end{cases}\]
In this case the global existence part of Theorem \ref{maintheorem1} follows by a continuity argument, and the decay estimates follow from an application of the pointwise estimates of Proposition \ref{prop:estimates} and these energy bounds.

\section{Global stability in the setting of {\textbf{SL}}}\label{section:negpos}
In this last section we use the tools from Section \ref{section:tools} to prove stability of the totally geodesic background $\varphi_{\text{I}}\circ\varphi_{\text{S}}$ in the case that the target has positive curvature, i.e. \eqref{waveu1SFver3} and \eqref{KGuiSFver3}.  With the notations introduced in the previous section,  we reduce our attention to 
\begin{equation} \label{possystem}
\begin{aligned}
\Box_\M u^1 &= 2 \langle\vec{u}, \p_{x^1}\vec{u}  \hspace{.05cm} \rangle + (\vec{u}\hspace{.05cm})^3 +  \vec{u} \cdot \M(\D \bm u, \D \bm u) + (\vec u\hspace{.05cm})^2 \cdot \p_{x^1} \bm u, \\
\Box_\M u^i - u^i & = -2u^i \p_{x^1} u^1 + (\vec{u}\hspace{.05cm})^3 + \vec{u}\cdot \M(\D \bm u, \D \bm u) + (\vec u\hspace{.05cm})^2 \cdot \p_{x^1} \bm u, \quad i = 2,\dots,n
\end{aligned}
\end{equation}
With $\p_{x^1}$ replaced by $\p_t$ on the right hand side, the system above is the same with the negative curvature case \eqref{negsystem}. Employing
 \beq\label{re-p_x-L}
\p_{x^i}=\frac{1}{t}L^i -\frac{x^i}{t} \p_t,
\eeq 
 what we can prove is:
\begin{thm}\label{maintheorem2}
Under the same assumptions, the results of Theorem \ref{maintheorem1} also apply to the system \eqref{possystem}.
\end{thm}
\begin{proof}
It suffices to obtain similar estimates as those in Propositions \ref{quadenergy} and  \ref{cubicenergy}, then the theorem \ref{maintheorem2} follows similarly from Corollary \ref{cor-apri-est} and Proposition \ref{prop:bootstrap}.
We first deal with the quadratic terms.  We decompose from \eqref{re-p_x-L}  the quadratic terms into two parts, i.e. $Q(m)=Q_1(m)+Q_2(m)$ with
\beqs
\begin{aligned}
Q_1(m)&\eqdef\int_{\Sigma_\tau} \left[ 2 L^\alpha \left< \vec u, t^{-1} L^1 \vec u \right> \p_t L^\alpha u^1- 2\left<L^\alpha (\vec u\,  t^{-1} L^1u^1), \p_t L^\alpha \vec u\right> \right] \dvol_{\Sigma_\tau};\\
Q_2(m)&\eqdef \int_{\Sigma_\tau} \left[ -2 L^\alpha \left< \vec u, t^{-1} x^1\p_t  \vec u \right> \p_t L^\alpha u^1+ 2\left<L^\alpha (\vec u\,  t^{-1} x^1 \p_t u^1), \p_t L^\alpha \vec u\right> \right] \dvol_{\St}
\end{aligned} 
\eeqs 
for an $m$-tuple $\alpha$ with entries in $\{1,2,3\}$.  The $Q_2$ term has the same structure with the quadratic nonlinearities for the negative curvature case presented in previous section, with the introduction of the factor $t^{-1}x^1 = w_\tau^{-1} x^1$.  In particular, the top order terms can be canceled. As we will see, the boosts $L^i$ acting on $t^{-1}x^1$ only contribute lower order terms because of \eqref{derweights}.

On the other hand, the top order of $Q_1(m)$ can not be cancelled but we can utilize the extra decay of $t^{-1}$. We claim the quadratic terms can be bounded as
\beq\label{claim-Q}
\begin{aligned}
|Q(m)|\lesssim \left\{\begin{array}{cc}
\tau^{-3/2} \left(  \fk E_2(\tau)\fk E_m(\tau)^2\right),  \qquad &\hbox{if $m=0,1$};\\
\tau^{-1}\fk E_1(\tau) \fk E_2(\tau)^2 + \tau^{-3/2}\fk E_2(\tau)^3, \qquad & \hbox{if $m=2$}.
\end{array}
\right.
\end{aligned}
\eeq
Besides the inequality $w_\tau^{-1}\lesssim \tau^{-1}$,  we shall  also use the identity $t=w_\tau$ on the surface $\Sigma_\tau$. As $Q_2$ can be dealt with in the same way as Proposition \eqref{quadenergy}, we only provide the proof for $Q_1(m)$. The case $m = 0$ is straightforward as we control $\vec u$ using the pointwise estimates of Proposition \ref{prop:estimates} and the rest of the vectors using H\"older's inequality. For the case $m=1$ we need to estimate
\beq\label{est-Q_1-1}
 \int_{\Sigma_\tau} \left(w_\tau^{-1}|L\vec u\h{1.5pt}|^2+ w_\tau^{-1}|\vec u||LL\bm u|+|L^\alpha (t^{-1})|  |\vec u \h{1.5pt}||L \bm u|\right) |\p_t L\bm  u| \dvol_\St.
\eeq
 Choosing the weights appropriately  and applying H\"older's inequality imply 
\beqs
\begin{aligned}
\int_{\Sigma_\tau} w_\tau^{-1} |L\vec u\h{1.5pt}|^2 |\p_t L \bm u| \dvol_{\Sigma_\tau} & \lesssim   \tau^{-3/2} \fk E_2(\tau) \int_{\Sigma_\tau } \, w_\tau^{1/2} |L\vec u\h{1.5pt}| w_\tau^{-1/2}|\p_t L \bm u|\dvol_\St\\
%&\lesssim \tau^{-3/2} \left( \int_{\Sigma_\tau} |L\vec u\h{1.5pt}|^2 w_\tau \dvol_{\Sigma_\tau}\right)^{1/2} \left(\int_{\Sigma_\tau} |\p_t L\bm u|^2 \dvol_{\Sigma_\tau}\right)^{1/2}\\
&\lesssim \tau^{-3/2} \fk E_2(\tau) \fk E_1(\tau)^2.
\end{aligned}
\eeqs 
Here we also bounded $L^\infty$-norm of $L\vec u$ through  Proposition \ref{prop:estimates}. With a use of the definition of $L^i$ we have
\beqs
L^i(t^{-1})= -\frac{x^i}{t^2}
\eeqs
which on the surface $\St$ admits an upper bound $w_\tau^{-1}$, and hence implies
\beqs
\begin{aligned}
 \int_{\Sigma_\tau} |L^\alpha (t^{-1})|  |\vec u\h{1.5pt}||L \bm u| |\p_t L\bm  u| \dvol_\St&\lesssim  \tau^{-3/2}\fk E_2(\tau) \int_\St w_\tau^{-1/2} |L\bm u| w_\tau^{-1/2}|\p_t L\bm u| \dvol_\St\\
& \lesssim  \tau^{-3/2} \fk E_2(\tau)\fk E_1(\tau)^2.
 \end{aligned}
\eeqs
Here we also applied the $L^\infty$-bound of $\vec u$ in Proposition \ref{prop:estimates}. In a similar way the second term in \eqref{est-Q_1-1} admits the same upper bound which furthermore implies  \eqref{claim-Q} for $m=1$.

It remains to establish \eqref{claim-Q} for $m=2$, in which case $Q_1(m)$ can be bounded by
\beqs
\int_\St w_\tau^{-1}\left( |\vec u\h{1.5pt}||LL L \bm u|+  |LL\vec u\h{1.5pt}||L \bm u|+|L\vec u\h{1.5pt}||LL \bm u| \right) |\p_t LL \bm u|\dvol_\St + \text{ l.o.t},
\eeqs
where  the lower order terms are those that show up when $L$ acts on $t^{-1}$ resulting
\beq\label{est-L-t}
|L^\alpha (t^{-1})|\lesssim t^{-1}
\eeq 
for any $m$-tuple $\alpha$. It suffices to bound the top order terms. Bounding $L^\infty$-norm of $\vec u$ with the aid of Proposition \ref{prop:estimates} implies the first term can be bounded by
\beqs
\begin{aligned}
\int_\St w_\tau^{-1} |\vec u\h{1.5pt}| |LL L \bm u| |\p_t LL \bm u|\dvol_\St &\lesssim \tau^{-3/2} \fk E_2(\tau) \int_\St w_\tau^{-1/2}|LLL \bm u| w_\tau^{-1/2}|\p_tLL \bm u | \dvol_\St\\
&\lesssim \tau^{-3/2}  \fk E_2(\tau)^3.
\end{aligned}
\eeqs
The last two terms  can be dealt with by using the interpolation Sobolev inequalities in Proposition \ref{prop:estimates}. In particular, the second term can be bounded as
\beqs
\begin{aligned}
\int_\St w_\tau^{-1} |LL\vec u\h{1.5pt}||L \bm u| |\p_t LL \bm u|\dvol_\St & \lesssim \tau^{-1} \int_\St |LL\vec u\h{1.5pt}|w_\tau^{1/3} |L\bm u|w_\tau^{1/6} |\p_t LL \bm u|w_\tau^{-1/2} \dvol_\St\\
&\lesssim \tau^{-1}\|\vec u\h{1.5pt}\|_{\mr \Sobw_1^{2,3}}\|\bm u\|_{\mr \Sobw_1^{1,6}} \|\p_t \bm u\|_{\mr \Sobw_{-1}^{2,2}}\\
&\lesssim \tau^{-1} \fk E_1(\tau)\fk E_2(\tau)^2.
\end{aligned}
\eeqs
In a similar manner, the third term admits upper bound
\beqs
\begin{aligned}
\int_\St w_\tau^{-1} |L\vec u\h{1.5pt}||L L\bm u| |\p_t LL \bm u|\dvol_\St & \lesssim \tau^{-1} \int_\St |L\vec u\h{1.5pt}|w_\tau^{1/3} |LL\bm u|w_\tau^{1/6} |\p_t LL \bm u|w_\tau^{-1/2} \dvol_\St\\
&\lesssim \tau^{-1} \fk E_1(\tau)\fk E_2(\tau)^2,
\end{aligned}
\eeqs
which competes the proof of Claim \eqref{claim-Q}. 

The cubic terms can be dealt with similarly. In particular,  the cubic terms in \eqref{possystem} by employing \eqref{re-p_x-L} can be decomposed into two parts, writing as $\mcC(m)=\mcC_1(m)+\mcC_2(m)$ with
\beqs
\begin{aligned}
\mcC_1(m)&=\int_{\St}  L^\alpha \left[  (\vec u\hspace{.05cm})^2 \cdot t^{-1}L^1 \bm u   \right] \cdot \p_t L^\alpha \bm u \dvol_\St;\\
\mcC_2(m)&=\int_{\St}  L^\alpha \left[  (\vec{u}\hspace{.05cm})^3 +  \vec{u} \cdot \M(\D \bm u, \D \bm u) + (\vec u\hspace{.05cm})^2 \cdot t^{-1}x^1 \p_t \bm u   \right] \cdot \p_t L^\alpha \bm u \dvol_\St.
\end{aligned}
\eeqs
for an $m$-tuple $\alpha$ with entries in $\{1,2,3\}$. Again the second term $\mcC_2(m)$ admit  similar structure of cubic terms for the negative case and hence has the same bound as in Proposition \ref{cubicenergy}.  Here we recall $L^i$ acting on $t^{-1}x^1$, or $w_\tau^{-1} x^1$ on $\St$,  only contribute lower order terms by \eqref{derweights}. The first item $\mcC_1$ can be dealt with by utilizing the extra decay of $t^{-1}$.  We claim
\beq\label{est-C_1}
|\mcC_1(m)| \lesssim \left\{\begin{aligned}
&\tau^{-3} \fk E_2(\tau)^2 \fk E_m(\tau)^2, \qquad \qquad\qquad\qquad   &\hbox{if $m=0,1$};\\
%&\tau^{-3} \fk E_2(\tau)^2 \fk E_1(\tau)^2,  & \hbox{if $m=1$};\\
& \tau^{-2} \fk E_2(\tau)^3 \fk E_1(\tau) +\tau^{-3}\fk E_2(\tau)^4, &\hbox{if $m=2$}.
\end{aligned}\right.
\eeq
For $m=0$,   the estimate above is a direct result of $L^\infty$ bound of $\vec u$ in Proposition \ref{prop:estimates} and H\"older's inequality. For $m=1$, utilizing \eqref{est-L-t} and $t=w_\tau$ on $\St$ we need estimate
\beqs
 \int_\St w_\tau^{-1}\left(|L \vec u\h{0.05cm} ||\vec u\h{0.05cm}| |L \bm u| +|\vec u \h{0.05cm}|^2|LL \bm u|  + |\vec u\h{0.05cm}|^2 |L\bm u|\right)|\p_t L\bm u| \dvol_\St.
\eeqs
We bound the $L^\infty$-norm of $L\bm u$ and $\vec u$ as in Proposition \ref{prop:estimates}, apply $w_\tau^{-1}\leq \tau^{-1}$ and distribute the weight appropriately, arriving at an upper bound:
\beqs
 \tau^{-3}\fk E_2^2 \int_\St \left( w_\tau^{1/2}|L \vec u\h{0.05cm} | +w_\tau^{-1/2}|LL \bm u|  + w_\tau^{-1/2} |L\bm u|\right) w_\tau^{-1/2}|\p_t L\bm u| \dvol_\St.
\eeqs 
Applying H\"older's inequality  implies the estimate for $m=1$ in \eqref{est-C_1}. Lastly, for $m=2$ we need bound
\beqs
\begin{aligned}
 \int_\St w_\tau^{-1} \left(|LL \vec u\h{0.05cm}||\vec u\h{0.05cm}| |L\bm u|+ |L\vec u \h{0.05cm}|^2 |L \bm u|+|\vec u\h{.05cm}|^2 |LLL \bm u| \right)|\p_t LL\bm u|\dvol_\St+ \mathrm{l.o.t}.
\end{aligned}
\eeqs 
Utilizing $L^\infty$ bound of $\vec u$ and $L\bm u$ in Proposition \ref{prop:estimates}  and $w_\tau^{-1}\leq \tau^{-1}$ yields an upper bound
\beqs\begin{aligned}%|\mcC_1(m)|\big|_{m=2}&\lesssim \
\tau^{-2} \fk E_2^2 \int_\St  \left( w_\tau^{-1/2} |LL \bm u\h{0.05cm}|+ w_\tau^{1/2}|L\vec u \h{0.05cm}|+\tau^{-1} w_\tau^{-1/2}  |LLL \bm u| \right)w_\tau^{-1/2}|\p_t LL\bm u|\dvol_\St+ \mathrm{l.o.t}.
\end{aligned}
\eeqs 
Then \eqref{est-C_1} for the case $m=2$ follows directly by H\"older's inequality.
\end{proof}

{\bf{Acknowledgements---}} The authors extend their gratitude to Willie Wong for helpful and illuminating discussions and for a close reading of a preliminary version of this paper. L. Abbrescia would like to thank Thomas Walpuski for some useful references. L Abbrescia was supported by an NSF Graduate Research Fellowship (DGE-1424871).

\bibliographystyle{amsalpha}

\bibliography{GeoWM}

\end{document}